\documentclass[11pt,final]{elsarticle}
\usepackage{amsmath,amsthm,amsfonts,latexsym,amssymb,enumerate,color}
\usepackage{soul,graphicx}
\usepackage[color]{showkeys}
\definecolor{refkey}{rgb}{0,0,1}
\definecolor{labelkey}{rgb}{1,0,0}
\usepackage{graphicx}
\usepackage{epstopdf}
\usepackage{soul,cancel}
\makeatletter
\def\ps@pprintTitle{%
  \let\@oddhead\@empty
  \let\@evenhead\@empty
  \def\@oddfoot{\reset@font\hfil\thepage\hfil\today}
  \let\@evenfoot\@oddfoot
}
\makeatother

\newtheorem{theorem}{Theorem}
\newtheorem{lemma}[theorem]{Lemma}
\newtheorem{corollary}[theorem]{Corollary}

\newtheorem{example}[theorem]{Example}
\newtheorem{prop}[theorem]{Proposition}

\newcommand{\re}{\operatorname{Re}}
\newcommand{\im}{\operatorname{Im}}
\newcommand{\tr}{\operatorname{Tr}}
\newcommand{\diag}{\operatorname{diag}}
\newcommand{\rank}{\operatorname{rank}}
\newcommand{\Span}{\operatorname{Span}}
\newcommand{\norm}[1]{\Vert#1\Vert}
\newcommand{\scal}[1]{\langle#1\rangle}
\newcommand{\abs}[1]{\lvert#1\rvert}
\newcommand{\eq} [1] {\begin{equation}\label{#1}\quad}
\newcommand{\en} {\end{equation}}
\newcommand{\C}{\mathbb{C}}
\newcommand{\R}{\mathbb{R}}
\numberwithin{equation}{section}

\begin{document}

\begin{frontmatter}
\title{Numerical Ranges of 4-by-4 Nilpotent Matrices:\\ Flat Portions on the Boundary}

\author[fsu]{Erin Militzer}
\ead{militze@ferris.edu}

 \author[cp]{Linda J. Patton \corref{cor1}}
 \ead{lpatton@calpoly.edu}

 \author[wm,nyuad]{Ilya~M.~Spitkovsky}
 \ead{ilya@math.wm.edu, ims2@nyu.edu}

\author[au]{Ming-Cheng Tsai}
\ead{mctsai2@gmail.com}

 \cortext[cor1]{Corresponding author}


\address[fsu]{Mathematics Department, Ferris State University, Big Rapids, MI 49307}
 \address[cp]{Mathematics Department, California Polytechnic State University, San Luis Obispo, CA 93407}
 \address[wm]{Department of Mathematics, College of William and Mary, Williamsburg, VA 23187-8795}
\address[nyuad]{Division of Science, New York  University Abu Dhabi (NYUAD), Saadiyat Island,
P.O. Box 129188 Abu Dhabi, UAE}
\address[au]{Department of Mathematics and Statistics, Auburn University, Auburn, AL 36849}

\tnotetext[support]{The first three authors (EM, LJP, and IMS) started their work on this paper during the REUF workshop at the American Institute of Mathematics (Palo Alto, California) in the Summer of 2011, and their further collaboration with MCT was enabled by IMS visiting
National Sun Yat-sen University (Kaohsiung, Taiwan) in December of 2014. We gratefully acknowledge the support of the respective institutions.
IMS was also supported in part by the Plumeri Award for Faculty Excellence from the College of William and Mary and by Faculty Research funding from the Division of Science and Mathematics, New York University Abu Dhabi.}

\begin{abstract}In their 2008 paper Gau and Wu conjectured that the numerical range of a 4-by-4 nilpotent matrix has at most two
flat portions on its boundary. We prove this conjecture, establishing along the way some additional facts of independent interest. In particular, a full
description of the case in which these two portions indeed materialize and are parallel to each other is included.
\end{abstract}


\end{frontmatter}

\section{Introduction}

We consider the space $\C^n$ endowed with the standard scalar product $\scal{.,.}$ and the norm $\norm{.}$ associated with it.  Elements $x\in\C^n$ are $n$-columns; however, to simplify the notation
we will write them as $(x_1,\ldots,x_n)$. For an $n$-by-$n$
matrix $A$, its {\em numerical range}, also known as the {\em field of values}, is defined as
\[ F(A)=\{ \scal{Ax,x}\colon x\in\C^n, \ \norm{x}=1\}.  \]
The classical Toeplitz-Hausdorff theorem claims that the set $F(A)$ is always convex; this and other well-known properties of the numerical range are discussed in detail, e.g., in monographs \cite{GusRa,HJ1}.

As was observed by Kippenhahn (\cite{Ki}, see also the English translation \cite{Ki08}), $F(A)$ can be described in terms of the homogeneous polynomial
\eq{pa} p_A(u,v,w)=\det(uH+vK+wI), \en where \[ H=\frac{A+A^*}{2}:=\re A, \quad K=\frac{A-A^*}{2i}:=\im A.\]
Namely, $F(A)$ is the convex hull of the curve $C(A)$ dual (in projective coordinates) to
$p_A(u,v,w)=0$. So, it is not surprising that, starting with $n=3$,  the boundary $\partial F(A)$ of $F(A)$ may contain line segments (sometimes also called {\em flat portions}), even when the polynomial $p_A$ is irreducible.
For a unitarily irreducible $3$-by-$3$ matrix there is at most one such flat portion, as was first observed in the same paper \cite{Ki}, with constructive tests for its presence provided in \cite{KRS,RS05}.

The phenomenon of flat portions in higher dimensions was further studied in \cite{BS041}. For convenience of reference, we restate here Theorem~37 from \cite{BS041}.

\begin{theorem}\label{4x4} {\em [Brown-Spitkovsky]} Any $4$-by-$4$ matrix has at most $4$ flat portions on the boundary of its numerical range ($3$, if it is unitarily irreducible). \end{theorem}

Of course, the upper bounds may be lower if an additional structure is imposed on $A$. In this paper, we will tackle the case of nilpotent matrices. For reducible $4$-by-$4$ nilpotent matrices it is easy to see that the maximum possible number of flat portions is one; for the sake of completeness, this result is stated with a proof in Section~\ref{wk} (Proposition~\ref{p:red}). Unitarily irreducible $4$-by-$4$ nilpotent matrices were considered by Gau and Wu in \cite{GauWu081}, where in particular examples of such matrices $A$ with two flat portions on $\partial F(A)$ were given and it was also conjectured that three flat portions do not materialize. This conjecture was supported there by the following theorem, which is a special case of their result for $n$-by-$n$ matrices \cite[Theorem~3.4]{GauWu081}.

\begin{theorem}\label{GauWuCircular} {\em [Gau-Wu]} If $A$ is a $4$-by-$4$ nilpotent matrix which has a $3$-by-$3$ submatrix $B$ with $W(B)$ a circular disk centered at the origin, then there are at most two flat portions on the boundary of $F(A)$. \end{theorem}

In this paper we  prove the Gau-Wu conjecture. This is done in Section~\ref{s:proof}.
As a natural preliminary step, necessary and sufficient conditions for a $4$-by-$4$ nilpotent matrix to have at least one flat portion on the boundary of its numerical range are derived in Section~\ref{wk}. A special family of nilpotent matrices that is important for the proof of the main theorem is analyzed in Section~\ref{s:alt}. Section~\ref{s:two}  contains necessary and sufficient conditions for a nilpotent matrix to have two parallel flat portions on the boundary of its numerical range. In addition, we  show there that for a nilpotent $4$-by-$4$ matrix $A$ with two non-parallel flat portions on the boundary of $F(A)$ that are on lines equidistant from the origin, these are the only flat portions. The latter result is also used in Section~\ref{s:proof}, where in Theorem~\ref{twoflat} it is shown that if $A$ is a $4$-by-$4$ nilpotent matrix, then $ \partial F(A)$ contains at most 2 flat portions. The proof follows from an analysis of the locations of the singularities of the boundary generating curve \eqref{pa}. In the final Section~\ref{s:5x5} we use Theorem~\ref{twoflat} to tackle the case of
$5$-by-$5$ unitarily reducible matrices.

\section{Matrices with a flat portion on the boundary of their numerical range}\label{wk}

We start with an easy case of unitarily reducible matrices.
\begin{prop}\label{p:red} Let $A$ be an $4$-by-$4$ unitarily reducible nilpotent matrix. Then its numerical range $F(A)$ has at most one flat portion on the boundary.
\end{prop}
\begin{proof}
Let $A$ be unitarily similar to a direct sum $B_1\oplus\cdots\oplus B_k$, with $k>1$. The blocks $B_j$ are of course also nilpotent, and the following cases are possible.

{\sl Case 1.} $k=2$. If $B_1$ is a $3$-by-$3$ unitarily irreducible nilpotent matrix and $B_2=[0]$, then $F(A)=F(B_1)$. According to \cite[Theorem~4.1]{KRS}, $F(B_1)$ either has no flat portions on the boundary or exactly one such portion. Thus, so does $F(A)$. If there are two nilpotent $2$-by-$2$ blocks, then the numerical range of each block is a circular disk centered at the origin, and $F(A)$ is the largest of these disks and hence has no flat portion on its boundary.

{\sl Case 2.} $k=3$, that is, $B_1$ is a $2$-by-$2$ unitarily irreducible nilpotent matrix, while $B_2=B_3=[0]$. Then $F(A)=F(B_1)$ is again a circular disk centered at the origin, and there are no flat portions on its boundary.

{\sl Case 3.} $k=4$, implying that $A=0$, and $F(A)=\{0\}$. Hence there are no flat portions. \end{proof}

If $A$ is not supposed to be unitarily reducible, the situation becomes more complicated. Let us establish the criterion for at least one
flat portion to exist on $\partial F(A)$. To this end, some background terminology and information is useful.

First, recall the notion of an {\em exceptional
supporting line} of $F(A)$ which for an arbitrary matrix $A$ was introduced in \cite{LLS13}. Namely, let $\ell_\theta$ be the supporting line of $F(A)$
having slope $-\cot\theta$ and such that $e^{-i\theta}F(A)$ lies to the right of the vertical line $e^{-i\theta}\ell_\theta$. Then this supporting line
is exceptional (and, respectively, $\theta$ is an {\em exceptional angle}) if at least one $z\in\ell_\theta\cap F(A)$ is {\em multiply generated}, that
is, there exist at least two linearly independent unit vectors $x_j$ for which $\scal{Ax_j,x_j}=z$. For a given $A$, the angle $\theta$ is exceptional if and only if the hermitian matrix
$\re (e^{-i\theta}A)$ has a multiple minimal eigenvalue \cite[Theorem 2.1]{LLS13}; denote by $\mathcal L$ the respective eigenspace. The above mentioned
value $z$ is unique if and only if the compression of $\im (e^{-i\theta}A)$ (equivalently: $A$) onto $\mathcal L$ is a scalar multiple of the identity; $z$ is then called a {\em multiply generated round boundary point} of $F(A)$.

On the other hand, all points in the relative interior of a flat portion on the boundary of $F(A)$ are multiply generated. So, flat portions occur only on exceptional supporting lines, and for them to materialize it is necessary and sufficient that the the compression $A|{\mathcal L}$  of $A$ onto $\mathcal L$ is {\em not} a scalar multiple of the identity.

In our setting we will have to deal with 2-dimensional $\mathcal L$. The following test is useful in this regard.

\begin{prop}\label{quadformtest} Let $A$ be such that for some $\theta \in [0, 2 \pi)$ there exist two linearly independent vectors $y_1,y_2$ corresponding to the same
eigenvalue $\lambda$ of $\re (e^{-i\theta}A)$, and let ${\mathcal L}=\Span\{y_1,y_2\}$. Then the compression of $A$ onto $\mathcal L$
is a scalar multiple of the identity if and only if
\eq{eqn1evec} \scal{Ay_1,y_1}\norm{y_2}^2=\scal{Ay_2,y_2}\norm{y_1}^2 \en and
\eq{eqn2evec}
\scal{Ay_2,y_1}\norm{y_1}^2=\scal{y_2,y_1}\scal{Ay_1,y_1}. \en
\end{prop}
\begin{proof} Without loss of generality we may normalize the vectors $y_1,y_2$, dividing each of them by its length, and thus rewrite \eqref{eqn1evec},
\eqref{eqn2evec} in a slightly simpler form
\eq{1evec} \scal{Ay_1,y_1}=\scal{Ay_2,y_2}, \en
\eq{2evec}
\scal{Ay_2,y_1}=\scal{y_2,y_1}\scal{Ay_1,y_1}. \en
Observe also that \eqref{2evec} means exactly that \eq{3evec} \scal{A\widetilde{y}_2,y_1}=0,\en  where $\widetilde{y}_2$ is a unit vector in $\mathcal L$ orthogonal to $y_1$.
So, we just need to show that $A|{\mathcal L}$ is a scalar multiple of the identity if and only if \eqref{1evec} and \eqref{3evec} hold.

The {\sl necessity} of \eqref{1evec}, \eqref{3evec} is trivial, and even holds for an arbitrary subspace $\mathcal L$, not consisting of eigenvectors of $\re (e^{-i\theta}A)$. As for their
{\sl sufficiency}, note that $\re (e^{-i\theta}A)|{\mathcal L}$, being a scalar multiple of the identity, commutes with $\im (e^{-i\theta}A)|{\mathcal L}$. Thus, $A|{\mathcal L}$ is normal.
As such, condition \eqref{3evec} implies that the matrix of $A|{\mathcal L}$ with respect to the orthonormal basis $\{y_1,\widetilde{y}_2\}$ is diagonal. Consequently, $y_1$ is an eigenvector
of $A|{\mathcal L}$ corresponding to its eigenvalue $\mu=\scal{Ay_1,y_1}$, and the latter is an endpoint of $F(A|{\mathcal L})$. On the other hand, \eqref{1evec} shows that this value
is attained at two linearly independent unit vectors, $y_1$ and $y_2$. This is only possible if $A|{\mathcal L}=\mu I$. \end{proof}

We return now to the nilpotent matrix setting. Let us first establish the criterion for an exceptional supporting line to exist, independent of whether or not it contains a proper flat portion.

\begin{theorem}\label{th:except}Let $A$ be a $4$-by-$4$ nilpotent matrix. Then $F(A)$ has an exceptional supporting line if and only if $A$
is unitarily similar to
\eq{Ae} \alpha\begin{bmatrix}0 & a_1 & a_2 & a_3 \\ & 0 & a_4 & a_5 \\ & & 0 & a_6 \\ & & & 0 \end{bmatrix}, \en
where \eq{aj} \alpha\in\C, \quad \abs{a_j}\leq 1 \text{ for } j=1,2,3,\en  \begin{align}\label{condex1} \abs{a_4-\overline{a_1}a_2}^2 &
= (1-\abs{a_1}^2)(1-\abs{a_2}^2),\\ \label{condex2}
\abs{a_5-\overline{a_1}a_3}^2 &
= (1-\abs{a_1}^2)(1-\abs{a_3}^2),\\
\abs{a_6-\overline{a_2}a_3}^2 &
= (1-\abs{a_2}^2)(1-\abs{a_3}^2), \label{condex3} \end{align} and \eq{arg} \arg(a_6-\overline{a_2}a_3)=\arg(a_5-\overline{a_1}a_3)-\arg(a_4-\overline{a_1}a_2)
\mod 2\pi. \en  \end{theorem}
Note that all three arguments in \eqref{arg} are defined only if the inequalities in \eqref{aj} are strict. If this is not the case, we agree by convention that condition \eqref{arg} is vacuous.

\begin{proof}The result obviously holds for $A=0$. Indeed, then $A$ is in the form \eqref{Ae}, and every supporting line of $F(A)=\{0\}$ is exceptional.
So, in what follows we will suppose that $A\neq 0$.

{\sl Necessity.}  Suppose $A\neq 0$ is nilpotent, and (at least) one of the supporting lines of $F(A)$ is exceptional. Multiplying $A$
by a unimodular scalar if needed, we may without loss of generality suppose in addition that the exceptional supporting line is vertical. Let $d(\geq 0)$ be its
distance from the origin. Then $\re A+dI$ is a positive semi definite matrix with rank at most~2.

If $d=0$, then $\re A$ is positive semidefinite with zero trace, and thus zero diagonal. This is only possible if $\re A=0$. But then $\im A$
differs from $A$ by a scalar multiply only, and is therefore nilpotent along with $A$. Being hermitian, it is also zero. We arrive at a contradiction
with $A$ being non-zero, implying that $d>0$. Multiplying $A$ by another scalar, this time positive, we may without loss of generality suppose that
$d=1/2$, that is, $A+A^*+I$ is positive semi definite of rank at most~2. We will show that for such matrices the statement holds with $\alpha=1$.

To this end, use unitary similarity to put $A$ in upper triangular form \eqref{Ae} with $\alpha=1$, and observe that then
\eq{AA*} A+A^*+I=\begin{bmatrix} 1 & a_1 & a_2 & a_3 \\ \overline{a_1} & 1  & a_4 & a_5 \\ \overline{a_2 }& \overline{a_4} & 1 & a_6 \\
\overline{a_3} & \overline{a_5 } & \overline{a_6 } & 1 \end{bmatrix}. \en
The matrix in the right hand side of \eqref{AA*} is congruent to $[1]\oplus G$, where \eq{G} G=
\begin{bmatrix} 1-\abs{a_1}^2 & a_4-\overline{a_1}a_2 & a_5-\overline{a_1}a_3 \\
\overline{a_4}-a_1\overline{a_2} &  1-\abs{a_2}^2 & a_6-\overline{a_2}a_3 \\
\overline{a_5}-a_1\overline{a_3} & \overline{a_6}-a_2\overline{a_3} & 1-\abs{a_3}^2\end{bmatrix}.  \en
So, $G$ must be positive semi definite of rank at most 1. The former property implies the inequalities in \eqref{aj}, while due to the latter
the three $2$-by-$2$ principal minors of $G$ are equal to zero. This is equivalent to \eqref{condex1}--\eqref{condex3}.
In its turn, if $\abs{a_1}<1$, then due to \eqref{condex1}, \eqref{condex2} $G$ is congruent to
\eq{H3} [1-\abs{a_1}^2]\oplus \begin{bmatrix} 0 & w \\ \overline{w} & 0\end{bmatrix}, \en  where
\[ w= a_6-\overline{a_2}a_3-\frac{(a_5-\overline{a_1}a_3)(\overline{a_4}-a_1\overline{a_2})}{1-\abs{a_1}^2}. \]
So, in this case \[ \rank G = \begin{cases} 1 & \text{ if } w=0, \\ 3 & \text{ otherwise},\end{cases} \] which implies \eqref{arg}.

{\sl Sufficiency.} Without loss of generality, let $A$ be given by \eqref{Ae} with $\alpha=1$. Then \eqref{AA*} and \eqref{G} hold.

If $\abs{a_1}=1$, then \eqref{condex1}, \eqref{condex2} imply
\eq{H2} G=[0]\oplus\begin{bmatrix}
1-\abs{a_2}^2 & a_6-\overline{a_2}a_3 \\ \overline{a_6}-a_2\overline{a_3} & 1-\abs{a_3}^2\end{bmatrix}, \en
and the second summand in \eqref{H2} is singular due to \eqref{condex3}. So, the minimal eigenvalue $-1/2$ of $\re A$ has multiplicity 2.
For the case $\abs{a_1}<1$, the same conclusion follows from the congruence of $G$ and \eqref{H3}, since \eqref{condex1}--\eqref{arg}
imply $w=0$.  \end{proof}

Note that the exceptional supporting line $\ell$, the existence of which is established by Theorem~\ref{th:except}, is the vertical line $x=-1/2$ scaled by $\alpha$. Consequently, $\ell$ is at the distance
$\abs{\alpha}/2$ from the origin, and has the slope $-\cot\arg\alpha$.

Also, conditions \eqref{condex1}--\eqref{condex3} can be rewritten in an equivalent form
\[ a_4= \overline{a_1}a_2+r_1r_2e^{i\theta_1},\quad   a_5= \overline{a_1}a_3+r_1r_3e^{i\theta_2},\quad  a_6= \overline{a_2}a_3+r_2r_3e^{i\theta_3}, \]
where \[ r_j=\sqrt{1-\abs{a_j}^2}, \quad j=1,2,3,\]  and
\eq{theta} \theta_1=\arg(a_4-\overline{a_1}a_2), \quad \theta_2=\arg(a_5-\overline{a_1}a_3), \quad  \theta_3=\arg(a_6-\overline{a_2}a_3). \en
Consequently, the matrix \eqref{Ae} that satisfies the conditions in Theorem \ref{th:except} can be represented more explicitly as \eq{Ae1}
\alpha\begin{bmatrix}0 & a_1 & a_2 & a_3 \\ & 0 & \overline{a_1}a_2+r_1r_2e^{i\theta_1} & \overline{a_1}a_3+r_1r_3e^{i\theta_2} \\ &  & 0 & \overline{a_2}a_3+r_2r_3e^{i\theta_3} \\ & & & 0 \end{bmatrix}, \en
where $\theta_3=\theta_2-\theta_1.$
We will now use Proposition~\ref{quadformtest} to establish the additional conditions on $A$ under which a flat portion of $F(A)$ on $\ell$ actually materializes.

\begin{theorem}\label{oneflat}Let $A$ be unitarily similar to \eqref{Ae1}. In the notation introduced above, $F(A)\cap\ell$ is a proper line segment unless one of the following four conditions holds.

\smallskip
(i) $r_1 r_2 r_3 \ne 0$ and $\tau_1=\tau_2=0$, where
\eq{bdryval}
\begin{array}{l c l}
\tau_1= r_3(r_1^2+r_2^2-r_1^2r_2^2)(\overline{a}_1 a_3 e^{-i \theta_2}-a_1 \overline{a}_3 e^{i \theta_2})&+&r_2(r_1^2+r_3^2-r_1^2r_3^2)(a_1 \overline{a}_2e^{i \theta_1}-\overline{a}_1 a_2 e^{-i \theta_1}) \\ &+&r_1r_2 r_3 |a_1|^2 (a_2 \overline{a}_3 e^{i (\theta_2-\theta_1)} -\overline{a}_2 a_3 e^{-i (\theta_2-\theta_1)}),
\end{array}
\en
and
\eq{mixedterm}
\begin{array}{l c l}
\tau_2=\overline{a}_2a_3r_1(r_1^2+2r_2^2-2r_1^2r_2^2)-a_1\overline{a}_2^2a_3r_1^2r_2e^{i \theta_1}+\overline{a}_1 a_3r_2(-r_2^2-r_1^2+r_1^2r_2^2)e^{-i \theta_1}\\+a_1 \overline{a}_2r_1^2r_3|a_2|^2e^{i \theta_2} +r_1r_2r_3(1-2|a_1|^2 |a_2|^2)e^{i(\theta_2-\theta_1)}+\overline{a}_1a_2|a_1|^2r_2^2r_3e^{i(\theta_2-2\theta_1)}.\end{array}
 \en

\smallskip
 (ii) $r_1=0$, $r_2=r_3 \ne 0$, and $\arg(a_3)=\arg(a_2)+ \theta_3.$

 \smallskip
 (iii) $r_2=0$, $r_1=r_3 \ne 0$, and $\arg(a_3)=\arg(a_1)+ \theta_2+\pi.$

 \smallskip
 (iv) $r_3=0$, $r_1=r_2 \ne 0$, and $\arg(a_2)=\arg(a_1)+ \theta_1.$
\end{theorem}

\begin{proof} Without loss of generality we may suppose that $A$ is in the form \eqref{Ae1}, not just unitarily similar to it, and that $\alpha=1$. Then

$$\re A=\frac{1}{2}\left(\begin{array}{cccc}0 & a_1 & a_2 & a_3 \\\overline{a}_1 & 0 & \overline{a}_1a_2 +r_1r_2e^{i \theta_1}& \overline{a}_1a_3 +r_1r_3e^{i \theta_2} \\\overline{a}_2 & a_1\overline{a}_2+r_1r_2e^{-i \theta_1} & 0 &  \overline{a}_2a_3 +r_2r_3e^{i \theta_3}\\\overline{a}_3 & a_1 \overline{a}_3 +r_1r_3e^{-i \theta_2}& a_2 \overline{a}_3+r_2r_3e^{-i\theta_3} & 0\end{array}\right).$$

\smallskip
{\sl Case 1.} Let $r_1r_2r_3 \ne 0$. It is straightforward to check that the minimal eigenvalue $\lambda=-\frac{1}{2}$ of $\re A$ has multiplicity 2, and
\[y_1=(-a_2r_1+a_1r_2e^{i \theta_1},-r_2e^{i \theta_1},r_1,0), \quad y_2=(-a_3r_1+a_1r_3e^{i \theta_2},-r_3e^{i \theta_2},0,r_1)\] form a basis of the respective eigenspace. Moreover, $$Ay_1=(-a_1r_2e^{i \theta_1}+a_2r_1,r_1\overline{a}_1a_2+r_1^2r_2e^{i \theta_1},0,0)$$ and $$Ay_2=(-a_1r_3e^{i \theta_2}+a_3r_1,r_1\overline{a}_1a_3+r_1^2r_3e^{i \theta_2},r_1\overline{a}_2a_3+r_1r_2r_3e^{i (\theta_2-\theta_1)},0).$$  Therefore

\begin{align*}
\|y_1 \|^2&=2r_1^2+2r_2^2-2r_1^2r_2^2-a_1\overline{a}_2r_1r_2e^{i \theta_1}-\overline{a}_1a_2r_1r_2e^{-i \theta_1}, \\
\|y_2 \|^2&=2r_1^2+2r_3^2-2r_1^2r_3^2-a_1\overline{a}_3r_1r_3e^{i \theta_2}-\overline{a}_1a_3r_1r_3e^{-i \theta_2},  \\
\langle Ay_1, y_1 \rangle &=-r_1^2-r_2^2+r_1^2r_2^2+a_1\overline{a}_2r_1r_2e^{i \theta_1}, \\
\langle A y_2,y_2 \rangle & =-r_1^2-r_3^2+r_1^2r_3^2+a_1\overline{a}_3r_1r_3e^{i \theta_2}.
\end{align*}

After some simplification, $\|y_2 \|^2 \langle Ay_1, y_1 \rangle-\|y_1 \|^2 \langle Ay_2, y_2 \rangle$ becomes $r_1\tau_1$ with $\tau_1$ defined by \eqref{bdryval}. Hence condition \eqref{eqn1evec} is satisfied
if and only if $\tau_1=0$.

Next note that 

$$\langle y_2, y_1 \rangle = r_1^2a_3 \overline{a}_2-\overline{a}_2a_1r_1r_3e^{i\theta_2}-a_3 \overline{a}_1 r_1r_2e^{-i \theta_1}+(1+|a_1|^2)r_2r_3 e^{i(\theta_2-\theta_1)},$$

and

$$\langle Ay_2, y_1 \rangle =a_1r_3e^{i \theta_2} (\overline{a}_2r_1-\overline{a}_1 r_2e^{-i \theta_1}).$$

Now $\|y_1 \|^2 \langle Ay_2, y_1 \rangle-\langle y_2, y_1 \rangle \langle A y_1, y_1 \rangle$ simplifies to $r_1\tau_2$, so \eqref{eqn2evec}
holds exactly when $\tau_2=0$.

Therefore, by Proposition~\ref{quadformtest}, the line $x=-\frac{1}{2}$ will contain a proper line segment of $F(A)$ if and only if at least one of  $\tau_1$ or $\tau_2$ is nonzero. This agrees with the statement of the theorem.

{\sl Case 2.} At least two of $r_j$ are equal to zero, $j=1,2,3$. To be consistent with the statement of the theorem, we need to show that $F(A)\cap\ell$ is a proper line segment. But this is indeed so.
For example, if $r_1=r_2=0$, then it immediately follows that the vectors $y_1=(-a_1, 1, 0,0)$ and $y_2=(-a_2, 0, 1, 0)$ are linearly independent eigenvectors of $\re A$ corresponding to $-\frac{1}{2}$. Since $A y_1=(a_1, 0,0,0)$ and $A y_2=(a_2,0,0,0)$, both sides of \eqref{eqn1evec} equal $-2$. However, \eqref{eqn2evec} is not satisfied because

$$\langle A y_2, y_1 \rangle \|y_1 \|^2=-2 a_2 \overline{a}_1 \ne- a_2 \overline{a}_1=\langle y_2, y_1 \rangle \langle A y_1, y_1 \rangle.$$

Therefore, the flat portion will exist. All other cases where at least two $r_j$ values are zero are treated in the same manner.

{\sl Case 3.} Exactly one of $r_j$ is equal to zero. For the sake of definiteness, let $r_1=0$, $r_2r_3\neq 0$.
Then
$y_1=(-a_1, 1,0,0)$ and $y_2=(a_2r_3e^{i\theta_3}-r_2 a_3,0,-r_3e^{i\theta_3},r_2)$ are linearly independent eigenvectors of $\re A$ corresponding to $-\frac{1}{2}$. It still holds that $\|y_1 \|^2=2$ and $\langle A y_1, y_1 \rangle=-1$.  Now we also have

$$\langle y_2, y_1 \rangle=r_2 a_3 \overline{a}_1-a_2 \overline{a}_1 r_3 e^{i\theta_3},$$

$$\|y_2 \|^2=2-\overline{a}_2 a_3 r_2 r_3 e^{-i\theta_3}-a_2 \overline{a}_3 r_2 r_3 e^{i\theta_3}-2|a_2|^2 |a_3|^2.$$

and  $$A y_2=(-a_2r_3e^{i\theta_3}+r_2a_3,-\overline{a}_1a_2r_3e^{i\theta_3}+\overline{a}_1a_3r_2, r_2\overline{a}_2a_3+r_2^2r_3e^{i\theta_3},0).$$

Therefore, we can can compute the remaining quantities from Proposition~\ref{quadformtest}:

$$ \langle A y_2, y_2 \rangle=a_2 \overline{a}_3 r_2 r_3 e^{i\theta_3}-1+|a_2|^2|a_3|^2. $$

and

$$ \langle A y_2, y_1 \rangle=0.$$

Equation \eqref{eqn1evec} holds if and only if $-\|y_2 \|^2=2  \langle A y_2, y_2 \rangle$. Substituting into the latter equation yields

$$ -2+\overline{a}_2 a_3 r_2 r_3 e^{-i\theta_3}+a_2 \overline{a}_3 r_2 r_3 e^{i\theta_3}+2|a_2|^2 |a_3|^2=2a_2 \overline{a}_3 r_2 r_3 e^{i\theta_3}-2+2|a_2|^2|a_3|^2,$$

which simplifies to

\begin{equation}\label{r1zeroeq1}
\overline{a}_2 a_3 r_2 r_3 e^{-i\theta_3}-a_2 \overline{a}_3 r_2 r_3 e^{i\theta_3}=0.
\end{equation}

Since $a_1 \ne 0$, equation \eqref{eqn2evec} holds if and only if

\begin{equation}\label{r1zeroeq2}
r_2 a_3 -a_2  r_3 e^{i\theta_3}=0.
\end{equation}

If equation \eqref{r1zeroeq2} holds, then $r_3=r_2$ and hence $\arg(a_3)=\arg(a_2)+\theta_3.$ This proves the necessity of the conditions in (ii) in order for $F(A)$ to fail to have a flat portion on $x=-\frac{1}{2}$. Conversely, if the conditions in (ii) hold, then clearly \eqref{r1zeroeq1} and \eqref{r1zeroeq2} are true, which results in no flat portion by Proposition~\ref{quadformtest}.
This agrees with (ii) in the statement of the theorem.

The situations when $r_2=0$, $r_1r_3\neq 0$ and $r_3=0, r_1r_2\neq 0$ can be treated similarly. The only difference will be in the specific choice of the eigenvectors, namely, \[ y_1=(-a_2, 0,1,0), \quad  y_2= (-a_3 r_1+r_3a_1e^{i \theta_2},-r_3e^{i \theta_2},0,r_1) \] in the former, and
\[ y_1=(-a_3, 0,0,1), y_2=(-a_2r_1+a_1r_2e^{i \theta_1},-r_2e^{i \theta_1}, r_1,0) \] in the latter. Direct computations show that a flat portion does not materialize if and only if, respectively, (iii) or (iv) holds.
\end{proof}

{Conditions (i)--(iv) of Theorem~\ref{oneflat} simplify somewhat if the entries $a_j$ in \eqref{Ae} are real. Indeed, then $\theta_j=0\mod\pi$, and the argument conditions in (ii)--(iv) boil down to
\[ a_2a_3(a_6-a_2a_3)\geq 0, \quad a_1a_3(a_5-a_1a_3)\leq 0, \text{ and }  a_1a_2(a_4-a_1a_2)\geq 0, \] respectively. In its turn, $\tau_1$ in (i) is zero automatically, while
\eq{realtau2}\begin{array}{lcl} \tau_2= r_1a_2a_3(r_1^2+2r_2^2-2r_1^2r_2^2)-
a_1r_2a_3(r_2^2+2r_1^2-2r_1^2r_2^2)e^{i\theta_1} \\ +
a_1a_2r_3(r_1^2+r_2^2-2r_1^2r_2^2)e^{i\theta_2}+
r_1r_2r_3(1-2a_1^2a_2^2)e^{i\theta_3}.\end{array}\en

\begin{example} Let $A$ be of the form \eqref{Ae} with  $\alpha=1$ and $a_1=\frac{\sqrt{2+\sqrt{3}}}{2}$, $a_2=\frac{1}{2}$, $a_3=\frac{\sqrt{2}}{2}$. Choosing $a_4,a_5,a_6$ in such a way that \eqref{condex1}--\eqref{condex3} hold with $\theta_j=0$ in \eqref{theta}, $j=1,2,3$ yields $a_4=\frac{\sqrt{2}}{2}$, $a_5=\frac{\sqrt{3}}{2}$, and $a_6=\frac{\sqrt{2+\sqrt{3}}}{2}$. A direct substitution into \eqref{realtau2} reveals that $\tau_2=0$.  So, $F(A)$ has no flat portion
on the  supporting line $x=-\frac{1}{2}$ by Theorem~\ref{oneflat} even though this line is exceptional by Theorem~\ref{th:except}.

\end{example}

\begin{figure}[h]
\centering
\includegraphics[scale=.6]{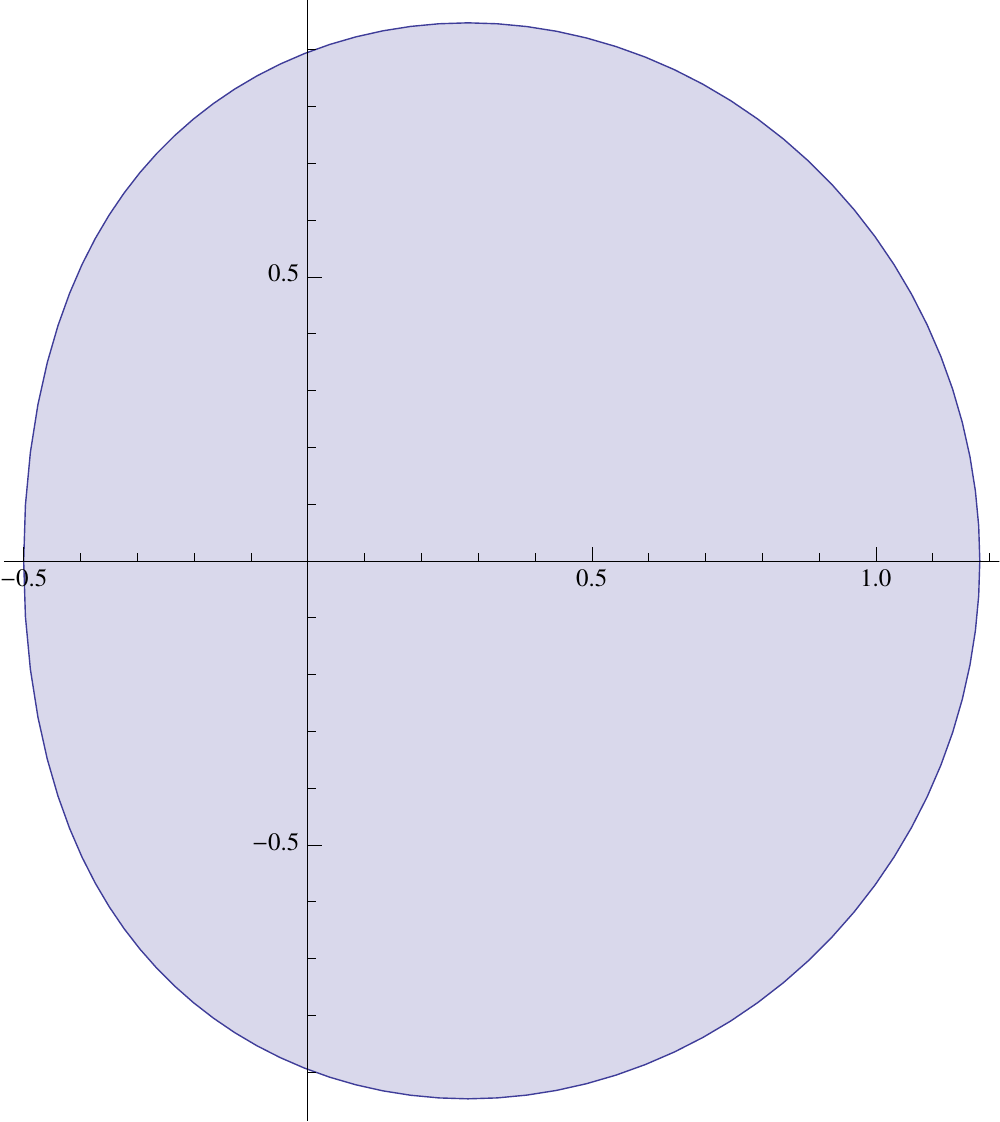}
\caption{$F(A)$ with exceptional support line but no flat portion}\label{noflat}
\end{figure}

\section{Special case: An alternative approach} \label{s:alt}

{As can be seen from the discussion above, Theorem~\ref{oneflat} provides a convenient tool for constructing specific examples of nilpotent matrices $A$ with a prescribed exceptional supporting line $\ell$, with
$F|(A)\cap\ell$ being just one point or a proper line segment. For another example, by setting $a_1=1$, $a_2=\frac{1}{2}$, $a_3=\frac{\sqrt{3}}{2}$, and $\theta_3=0$ in Theorem~\ref{th:except} we immediately obtain a matrix

\begin{equation*}
A=\left(\begin{array}{cccc}0 & 1 & \frac{1}{2} & \frac{\sqrt{3}}{2} \\0 & 0 & \frac{1}{2} & \frac{\sqrt{3}}{2} \\0 & 0 & 0 & \frac{\sqrt{3}}{2} \\0 & 0 & 0 & 0\end{array}\right)
\end{equation*}
that fails to satisfy any of the conditions (i)-(iv) in Theorem~\ref{oneflat} and hence has a flat portion as shown in Figure~\ref{withflat}.

\begin{figure}[h]
\centering
\includegraphics[scale=.6]{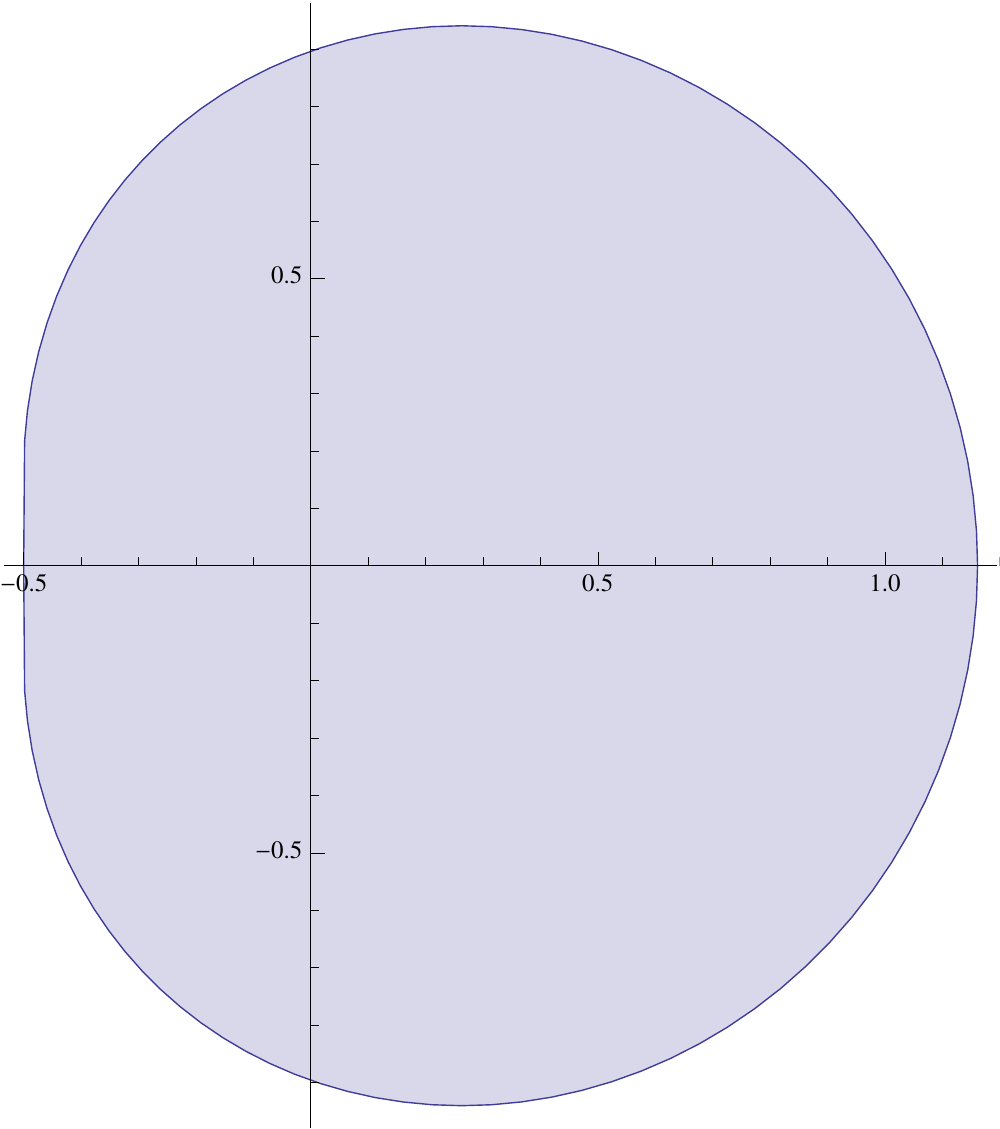}
\caption{$F(A)$ with flat portion}\label{withflat}
\end{figure}

 However, the conditions in Theorem~\ref{oneflat} are not  as useful when considering a given (even triangular) matrix
with the number and orientation of flat portions not known a priori. We present now one such case, to illustrate this point, and also since it plays an important role in Section~\ref{s:proof}.}

\begin{lemma}\label{realA} Let \[   A=\begin{bmatrix}0 & a_{1} & a_2 & a_{3} \\ 0 & 0 & a_{3} & a_2 \\  0 & 0 & 0 & a_{1} \\ 0 & 0 & 0 & 0\end{bmatrix}, \] where $a_1$, $a_2$, and $a_3$ are real and $a_1 \ne 0$. The boundary of $F(A)$ contains a vertical flat portion if and only if $|a_1|=|a_3|$ and $|a_2| \geq |a_1|$. In this case, this is the only flat portion on $\partial F(A)$.
\end{lemma}

\begin{proof} As is well-known (see, for example, \cite{GauWu08}), the boundary $\partial F(A)$ of the numerical range of $A$ contains a portion of the line $\cos(\theta) x+ \sin(\theta) y=d$ if and only if $d$ is the maximum (or minimum) eigenvalue of $\mathrm{Re} \, e^{-i \theta} A$ and there are two eigenvectors $y_1$ and $y_2$ associated with $d$ such that either condition \eqref{eqn1evec} or \eqref{eqn2evec} in Proposition~\ref{quadformtest} fails to hold.

A straightforward calculation shows that the characteristic polynomial of $\mathrm{Re} \, e^{-i \theta} A$ is

 \begin{equation*}
 q_{\theta}(\lambda)=\lambda^4-\frac{1}{2}(a_1^2+a_2^2+a_3^2) \lambda^2-(a_1a_2a_3 \cos\theta ) \lambda+\frac{1}{16}(a_1^4+a_2^4+a_3^4-2a_1^2a_2^2-2a_2^2a_3^2-2a_1^2a_3^2 \cos2 \theta ).
 \end{equation*}

This polynomial factors as

\begin{equation*}
 q_{\theta}(\lambda)=\left(\lambda^2+a_2 \lambda+ \frac{a_1a_3\cos\theta}{2}+\frac{a_2^2}{4}-\frac{a_1^2}{4}-\frac{a_3^2}{4}\right)\left(\lambda^2-a_2 \lambda- \frac{a_1a_3\cos\theta}{2}+\frac{a_2^2}{4}-\frac{a_1^2}{4}-\frac{a_3^2}{4}\right)
\end{equation*}
Therefore the roots of $q_{\theta}$, which are the eigenvalues of $\mathrm{Re} \, e^{-i \theta} A$, can be explicitly calculated as

\begin{equation}\label{evalsrotatedrealA} \begin{split}
\lambda_1(\theta) & =\frac{1}{2}\left(-a_2-\sqrt{a_1^2+a_3^2-2a_1a_3 \cos\theta}\right),\\
\lambda_2(\theta) & =\frac{1}{2}\left(a_2-\sqrt{a_1^2+a_3^2+2a_1a_3 \cos\theta}\right),\\
\lambda_3 (\theta)& =\frac{1}{2}\left(-a_2+\sqrt{a_1^2+a_3^2-2a_1a_3 \cos\theta}\right),\\
\lambda_4(\theta) & =\frac{1}{2}\left(a_2+\sqrt{a_1^2+a_3^2+2a_1a_3 \cos\theta}\right).
\end{split}\end{equation}

When $\theta=0$, the eigenvalues of $\mathrm{Re}A$ simplify to

\begin{equation*}\label{evalsrealA} \begin{split}
\lambda_1& =\frac{1}{2}\left(-a_2-a_1+a_3 \right),\\
\lambda_2 & =\frac{1}{2}\left(a_2-a_1-a_3 \right),\\
\lambda_3& =\frac{1}{2}\left(-a_2+a_1-a_3\right),\\
\lambda_4 & =\frac{1}{2}\left(a_2+a_1+a_3\right),
\end{split}\end{equation*}

where relabeling may have occurred based on absolute values. Furthermore, in this case it is straightforward to verify that corresponding eigenvectors of $\mathrm{Re} A$ are $ y_1=(1,-1,-1,1)$,  $y_2=(-1,1,-1,1)$, $y_3=(-1,-1,1,1)$, and $y_4=(1,1,1,1).$
Notice that

\begin{equation*}\label{Ay} \begin{split}
A y_1& =(-a_1-a_2+a_3,-a_3+a_2,a_1,0),\\
A y_2&=(a_1-a_2+a_3,a_2-a_3,a_1,0),\\
A y_3&=(-a_1+a_2+a_3,a_3+a_2,a_1,0),\\
A y_4&=(a_1+a_2+a_3,a_2+a_3,a_1,0).
\end{split}\end{equation*}

The values $\| y_j \|^2=4$ and $\langle A y_j, y_j \rangle=4 \lambda_j$ for $j=1,2,3,4.$
Therefore for all $1 \leq j, k \leq 4$,  if there is a repeated eigenvalue $\lambda_j=\lambda_k$ (extremal or not), we have

$$ \langle A y_j, y_j \rangle \|y_k \|^2=4 \lambda_j=4 \lambda_k= \langle A y_k, y_k \rangle \|y_j \|^2.$$
Hence \eqref{eqn1evec} will always hold, and a vertical flat portion will exist exactly when there is an extremal eigenvalue where \eqref{eqn2evec} fails. When $j \ne k$, the eigenvectors $y_j$ and $y_k$ are orthogonal. Accordingly,  \eqref{eqn2evec} fails to hold for a case where $\lambda_j=\lambda_k$ if and only if

\begin{equation}\label{eqn3evec}
\langle A y_j, y_k \rangle \ne 0.
\end{equation}

All possible cases to consider can be studied with the following inner products.

\begin{align}
\label{IPcase1} \langle A y_1, y_2 \rangle & =2a_3-2a_2,\\
\label{IPcase2} \langle A y_3, y_4 \rangle & =-2a_3-2a_2,\\
 \label{IPcase3} \langle A y_1, y_3 \rangle&=\langle A y_2, y_4 \rangle  =-2a_2,\\
\label{IPcase4} \langle A y_1, y_4 \rangle&=\langle A y_2, y_3 \rangle =0 \\
\nonumber
\end{align}

Now, assume the boundary of $F(A)$ has a vertical flat portion. Thus the maximal or minimal eigenvalue of $\mathrm{Re} A$ is repeated and the corresponding eigenvectors satisfy \eqref{eqn3evec}.

The equality $\lambda_1=\lambda_2$ implies $a_2=a_3$, so equation \eqref{IPcase1} shows condition \eqref{eqn3evec} fails in this case. Similarly, by equation \eqref{IPcase2}, condition \eqref{eqn3evec} fails when $\lambda_3 =\lambda_4$ and $a_2=-a_3$. Clearly, by \eqref{IPcase4}, no combination with $\lambda_1=\lambda_4$ or $\lambda_2=\lambda_3$ is possible with a vertical flat portion.

Therefore it must hold that either $\lambda_1=\lambda_3$ or $\lambda_2=\lambda_4$. In the former case, we have $a_1=a_3$; in order for this repeated $\lambda_1$ to be extremal, we must have $|a_2| \geq |a_1|$. Similarly, if $\lambda_2=\lambda_4$, then $a_1=-a_3$ and to be extremal $|a_2| \geq |a_1|$. Thus a vertical flat position implies the conditions  $|a_1|=|a_3|$ and $|a_2| \geq |a_1|$.

Conversely, if $a_1 \ne 0$, $a_1=a_3$ and $|a_2| \geq |a_1|$, then $\lambda_1=\lambda_3$ is either the maximal or minimal eigenvalue of $\mathrm{Re}A$ and $\langle A y_1, y_2 \rangle \ne 0$ by \eqref{IPcase3}.
Likewise, if $a_3=-a_1 \ne 0$ and $|a_2| \geq |a_1|$, then $\lambda_2=\lambda_4$ is an extreme eigenvalue and $\langle A y_2, y_4 \rangle \ne 0$, again by \eqref{IPcase3}. Hence in both cases Proposition~\ref{quadformtest} shows that $\partial F(A)$ has a vertical flat portion.

There are never two vertical flat portions on $\partial F(A)$ because that would require both a repeated maximal and a repeated minimal eigenvalue of $\mathrm{Re} A$, which implies $a_1=a_3$ and $a_1=-a_3$ and we assumed $a_1 \ne 0$. To see that a vertical flat portion never coexists with any other flat portion, assume there is a vertical flat portion and there is also a flat portion on the line $\cos\theta x+ \sin\theta y=d$ for some real $d$ and $\theta \in (0, 2 \pi)$ with $\theta \ne \pi$. It suffices to show that there is not a repeated maximal eigenvalue at $\theta$, because if there was a repeated minimal eigenvalue at $\theta$ there would also be a repeated maximal eigenvalue at $\theta \pm \pi$. In the list of the eigenvalues of $\mathrm{Re} \, e^{-i \theta} A$ in \eqref{evalsrotatedrealA}, $\lambda_2(\theta)<\lambda_4(\theta)$ and $\lambda_1(\theta)<\lambda_3(\theta)$. Therefore the only possibility of a repeated maximal eigenvalue is when $\lambda_3(\theta)=\lambda_4(\theta)$.
Setting $|a_1|=|a_3|$ in this equality yields

$$2a_2=\sqrt{2a_1^2-2a_1^2 \cos\theta}-\sqrt{2a_1^2+2a_1^2\cos\theta}=2|a_1| \left( \sin\left(\frac{\theta}{2}\right)-\vert \cos\left(\frac{\theta}{2}\right) \vert \right).$$
The extreme values of $ \sin(\frac{\theta}{2})-\vert \cos(\frac{\theta}{2}) \vert $ on $[0, 2 \pi)$ are 1 (only when $\theta=\pi$) and -1 (only when $\theta=0$). Thus $\lambda_3(\theta)=\lambda_4(\theta)$ cannot hold at another value of $\theta$ or else $|a_2| \geq |a_1|$ is contradicted.

\end{proof}

A sketch of the numerical range of $A$ when $a_1=1$, $a_3=-1$, and $a_2=2$ is shown in Figure~\ref{verticalflat}.

\begin{figure}[h]
\centering
\includegraphics[scale=.6]{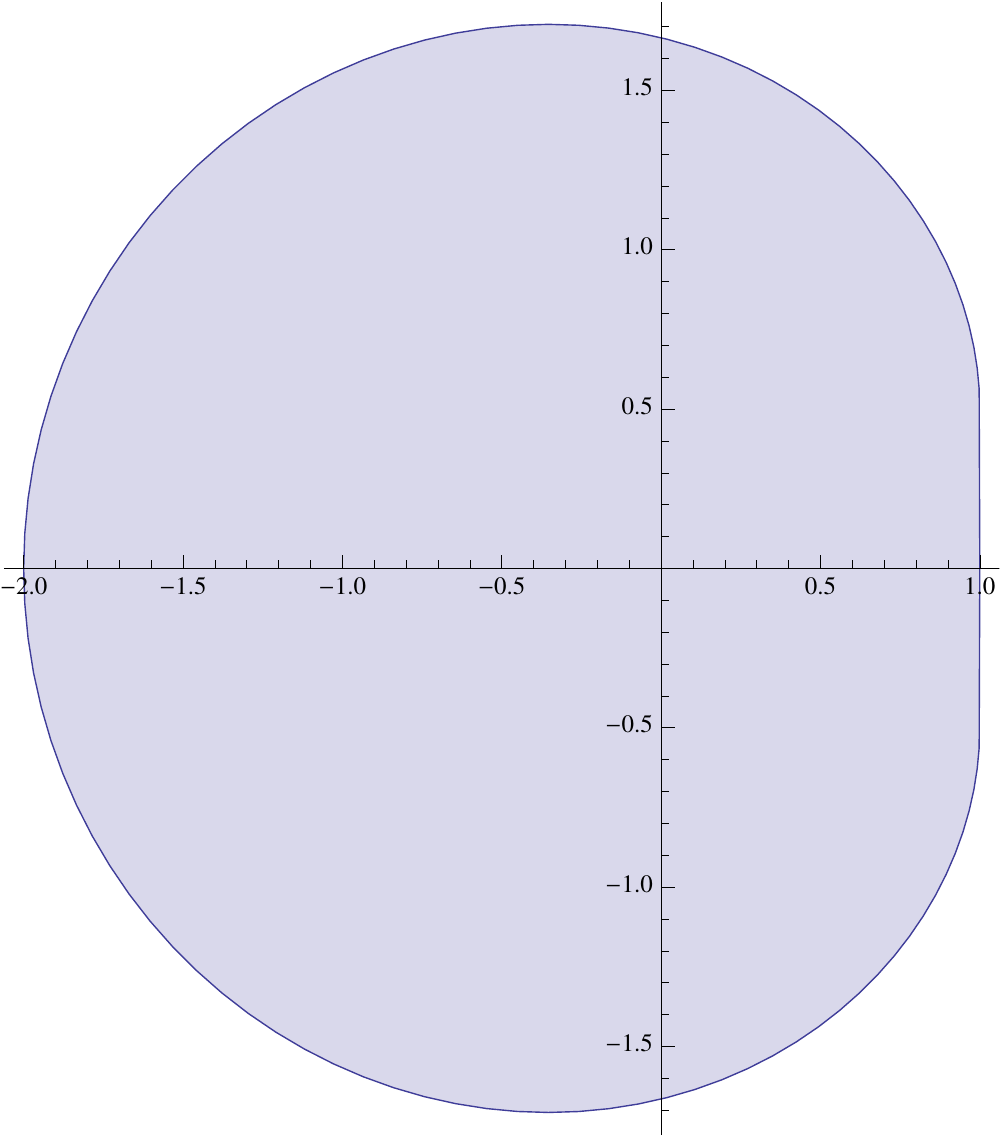}
\caption{$F(A)$ with one vertical flat portion}\label{verticalflat}
\end{figure}

\section{Matrices with two flat portions}\label{s:two}

\begin{prop}\label{p:par} A $4$-by-$4$ matrix $A$ is nilpotent, with two parallel flat portions on the boundary of its numerical range, if and only if it is a product of a scalar
$\alpha\in\C\setminus\{0\}$ and a matrix unitarily similar to
\eq{paralcanon} \begin{bmatrix}0 & a_1 & a_2 & a_3 \\ 0 & 0 & a_3 & -a_2 \\ 0 & 0 & 0 & a_1 \\ 0 & 0 & 0 & 0 \end{bmatrix}, \en
with $a_1,a_3>0$ and $a_2\in\R$.
\end{prop}
\begin{proof} {\sl Necessity.} Using Schur's lemma, we may without loss of generality suppose that the matrix $A$ is upper triangular. Being nilpotent, it thus can be written as
\eq{tri} A=\begin{bmatrix}0 & a_{12} & a_{13} & a_{14} \\ 0 & 0 & a_{23} & a_{24} \\ 0 & 0 & 0 & a_{34} \\ 0 & 0 & 0 & 0 \end{bmatrix}. \en
Note that the general form of $A$ has entries with double subscripts, but we will gradually simplify $A$ so that it has the form \eqref{paralcanon}. Multiplying matrix \eqref{tri} by an appropriate scalar (of absolute value one, if desired), we may also suppose that the parallel flat portions on the boundary of $F(A)$ are horizontal. Yet another unitary similarity, this time via a diagonal matrix, allows us to adjust the arguments of the entries $a_{i,i+1}$ any way we wish ($i=1,2,3$) without changing the absolute values. Let us agree therefore to choose these entries real and non-negative.

Having agreed on the above, we now introduce $H=\re A$ and $K=\im A$ as the hermitian matrices from the representation $A=H+iK$, so in particular
\eq{K} K=\frac{1}{2i}\begin{bmatrix}0 & a_{12} & a_{13} & a_{14} \\ -a_{12} & 0 & a_{23} & a_{24} \\ -\overline{a_{13}} & -a_{23} & 0 & a_{34} \\ -\overline{a_{14}} & -\overline{a_{24}} & -a_{34} & 0 \end{bmatrix}. \en
For a matrix $B$ of any size, two horizontal flat portions on the boundary of $F(B)$ can materialize only if the maximal and minimal eigenvalues of $\im B$ both have multiplicity at least 2. For our $4$-by-$4$ matrix $A$ it means that $K$ has two distinct eigenvalues, say $\lambda_1$ and $\lambda_2$, each of multiplicity two. In addition, $\tr K=0$, and so $\lambda_1=-\lambda_2 (:=\lambda)$. Consequently, \[ K^2=\lambda^2 I.\]
In particular, the diagonal entries of $K^2$ are all equal. From here and \eqref{K}:
\[ a_{12}^2+\abs{a_{13}}^2+\abs{a_{14}}^2=a_{12}^2+a_{23}^2+\abs{a_{24}}^2=a_{34}^2+\abs{a_{13}}^2+a_{23}^2=a_{34}^2+\abs{a_{14}}^2+\abs{a_{24}}^2. \]
Equivalently (and taking into account non-negativity of $a_{12},a_{23},a_{34}$):
\eq{fromdiag} a_{12}=a_{34}:=a_1, \quad \abs{a_{14}}=a_{23}:=a_3,\quad \abs{a_{13}}=\abs{a_{24}}.\en
Taking \eqref{fromdiag} into consideration, the fact that off diagonal entries of $K^2$ are all equal to zero boils down to
\eq{fromoffdiag} a_1(a_3-a_{14})=0, \quad a_1(a_{13}+a_{24})=0,\quad a_3a_{13}+a_{14}\overline{a_{24}}=0,\quad a_3a_{24}+\overline{a_{13}}a_{14}=0. \en
On the other hand, from \eqref{tri} and \eqref{fromdiag}:
\[ A^2= \begin{bmatrix} 0 & 0 & a_1a_3 & a_1(a_{13}+a_{24}) \\ 0 & 0 & 0 & a_1a_3 \\ 0 & 0 & 0 & 0 \\  0 & 0 & 0 & 0 \end{bmatrix}, \]
which, when combined with the second equation in \eqref{fromoffdiag}, yields
\[ A^2= a_1a_3\begin{bmatrix} 0 & 0 & 1 & 0 \\ 0 & 0 & 0 & 1 \\ 0 & 0 & 0 & 0 \\  0 & 0 & 0 & 0 \end{bmatrix}. \] But $A^2\neq 0$, since otherwise $F(A)$ would be a circular disk (see e.g. \cite{TsoWu}) exhibiting no flat portions on the boundary. So, $a_1,a_3>0$. The solution to \eqref{fromoffdiag} is then given by \[ a_{14}=a_3, \quad a_{13}=-a_{24}\ (:=a_2)\in\R. \]
So, $A$ is indeed in the form \eqref{paralcanon} up to unitary similarity and scaling.

{\sl Sufficiency.} The scalar multiple $\alpha$ is inconsequential, so without loss of generality $A$ is given by \eqref{paralcanon}. The eigenvalues of $K=\im A$ are then $\pm\lambda$, each of multiplicity two, where
$\lambda=\sqrt{a_1^2+a_2^2+a_3^2}/2$. Let us apply a unitary similarity, putting $K$ in the form $\begin{bmatrix}\lambda I & 0 \\ 0 & -\lambda I\end{bmatrix}$, and denote by $\begin{bmatrix}H_1 & Z \\ Z^* & H_2\end{bmatrix}$
the result of applying the same unitary similarity to $\re A$.

Since $A$ is real, its numerical range $F(A)$ is symmetric with respect to the $x$-axis, and thus there are either two or no horizontal flat portions on its boundary. But the latter is possible only if $H_1$ and $H_2$ are scalar multiples of the identity. Applying yet another (block diagonal) unitary similarity, we can reduce $A$ to \eq{red} \begin{bmatrix}h_1+i \lambda & 0 & \sigma_1 & 0 \\ 0 &h_1+i \lambda & 0 & \sigma_2 \\
 \sigma_1 & 0 &h_2-i \lambda & 0 \\ 0 & \sigma_2 & 0 & h_2-i \lambda \end{bmatrix}, \en where $\sigma_1,\sigma_2$ are the $s$-numbers of $Z$. The matrix \eqref{red} is unitarily (and even permutationally) reducible, which is in contradiction with the fact that $A$ has just one Jordan block. So, there are indeed two parallel flat portions on the boundary of $F(A)$.
\end{proof}
{\sl Remark.} From the proof of Proposition~\ref{p:par} it is clear that the parallel flat portions of $\partial F(A)$ are on lines that are at an equal distance $\abs{\alpha}\sqrt{a_1^2+a_2^2+a_3^2}/2$ from the origin, forming the angle $\arg\alpha$ with the positive direction of the $x$-axis. Of course, $\abs{\alpha}$ can be changed arbitrarily via absorbing it (or part thereof) by the $a_j$ entries in \eqref{paralcanon}.
\begin{corollary}\label{co:onlytwo}If $A$ is a $4$-by-$4$ nilpotent matrix with two parallel flat portions on the boundary of its numerical range, then these are the only flat portions of the boundary. \end{corollary}
\begin{proof} Without loss of generality, $A$ is in the form \eqref{paralcanon}, and thus there are two horizontal flat portions on the boundary of $F(A)$. We also know $A$ is irreducible from Proposition~\ref{p:red}. Any other flat portion $\ell$ of $\partial F(A)$, if it exists, cannot be horizontal. Suppose it is not vertical either. Then, due to the symmetry of $F(A)$ with respect to the $x$-axis, $\partial F(A)$ would have to contain the complex conjugate of $\ell$ as well, bringing the number of flat portions to (at least) four. But any $4$-by-$4$ matrix with four flat portions on the boundary of its numerical range is unitarily reducible, as stated in Theorem~\ref{4x4}, while $A$ is not.

It remains to consider the case of vertical $\ell$. In order for it to exist, the matrix $H=\re A$ should have a multiple eigenvalue. This eigenvalue would then have to be common for $3$-by-$3$ principal submatrices of $H$. A direct computation shows, however, that the characteristic polynomials of these submatrices up to a constant multiple equal
 \[ -4\lambda^3+\lambda (a_1^2+a_2^2+a_3^2)+a_1a_2a_3 \text{ and } -4\lambda^3+\lambda (a_1^2+a_2^2+a_3^2)-a_1a_2a_3. \] So, the common eigenvalues occur only if $a_2=0$ (recall that $a_1$ and $a_3$ are strictly positive).

On the other hand, if $a_2=0$, then the matrix $iA$ under the diagonal unitary similarity via $\diag[1,i,1,i]$ turns into
\[ \begin{bmatrix} 0 & a_1 & 0 & a_3 \\ 0 & 0 & -a_3 & 0 \\ 0 & 0 & 0 & a_1 \\  0 & 0 & 0 & 0\end{bmatrix}.\]
So, $F(A)$ in this case is symmetric with respect to the $y$-axis as well. Thus, a vertical flat portion $\ell$ would have its counterpart $-\ell$ on $\partial F(A)$ which again would bring the number of flat portions up to four, in contradiction with the unitary irreducibility of $A$. \end{proof}

\begin{prop}Let $A$ be a $4$-by-$4$ nilpotent unitarily irreducible matrix with two non-parallel flat portions on the boundary of its numerical range $F(A)$ the supporting lines of which are equidistant from the origin. Then these are
the only flat portions of $\partial F(A)$. \label{p:nonpa} \end{prop}
\begin{proof} Multiplying $A$ by a non-zero scalar we may without loss of generality suppose that the given flat portions of $\partial F(A)$ lie on lines intersecting at some point on the negative real half-axis and that the distance from each line to the origin equals $1/2$. Then for some unimodular $\omega=\xi+i\eta$ ($\xi \neq 0,\ \eta >0$) the imaginary part of both $\omega A$ and $-\overline{\omega}A$ will have  a multiple eigenvalue $-1/2$. With this notation, these lines will be $\eta x \pm \xi y=-\frac{1}{2};$ they intersect at $(-\frac{1}{2 \eta}, 0)$.

By an appropriate unitary similarity $A$ can be put in the upper triangular form \eqref{tri} and, moreover, the elements $a_{12}, a_{23}, a_{34}$ of its first sup-diagonal can be all made real. The above
mentioned condition on the eigenvalues of $\im (\omega A)$, $\im (-\overline{\omega} A)$ implies then
the matrices \eq{ima} \begin{bmatrix} 1 & -i\omega a_{12} & -i\omega a_{13} & -i\omega a_{14}\\ i\overline{\omega}\overline{a_{12}} & 1 & -i\omega a_{23} & -i\omega a_{24} \\
 i\overline{\omega}\overline{a_{13}} &  i\overline{\omega}\overline{a_{23}} & 1 & -i\omega a_{34}\\ i\overline{\omega}\overline{a_{14}} & i\overline{\omega}\overline{a_{24}} & i\overline{\omega}\overline{a_{34}} & 1\end{bmatrix}
 \text{ and }
\begin{bmatrix} 1 & i \overline{\omega} a_{12} & i \overline{\omega} a_{13} & i \overline{\omega} a_{14} \\ -i{\omega}\overline{a_{12}} & 1 & i \overline{\omega} a_{23} & i \overline{\omega} a_{24} \\
-i{\omega}\overline{a_{13}} &  -i{\omega}\overline{a_{23}} & 1 & i \overline{\omega} a_{34}\\ -i{\omega}\overline{a_{14}} & -i{\omega}\overline{a_{24}} & -i{\omega}\overline{a_{34}} & 1\end{bmatrix}\en
have rank at most 2. Equating the left upper $3$-by-$3$ minors of \eqref{ima} to zero, we see that  \[ 1-(\abs{a_{13}}^2+\abs{a_{12}}^2+\abs{a_{23}}^2)+2\im(\omega a_{12}\overline{a_{13}}a_{23})=0 \]
and \[ 1-(\abs{a_{13}}^2+\abs{a_{12}}^2+\abs{a_{23}}^2)-2\im(\overline{\omega}a_{12}\overline{a_{13}}a_{23})=0. \] Equivalently, $a_{12}\overline{a_{13}}a_{23}$ is real, and  \eq{1} 1-(\abs{a_{13}}^2+\abs{a_{12}}^2+\abs{a_{23}}^2)+2\eta a_{12}\overline{a_{13}}a_{23}=0.\en
If $a_{12}=0$ or $a_{23}=0$, then \eqref{1} implies that for the $3$-by-$3$ matrix $B$ located in the upper left corner of $A$ the numerical range $F(B)$ is the circular disk $\{z\colon\abs{z}\leq 1/2\}$
(see \cite{Marc} or \cite[Theorem 4.1]{KRS}). Since this case is covered by Theorem~\ref{GauWuCircular}, we may suppose that $a_{12},a_{23}\neq 0$. Then $a_{13}$ is necessarily real along with $a_{12},a_{23}$,
and \eqref{1} can be rewritten as
\eq{2} 1-(a_{13}^2+a_{12}^2+a_{23}^2)+2\eta a_{12}a_{13}a_{23}=0.\en

Repeating this reasoning for three other principal $3$-by-$3$ minors in \eqref{ima} we see that without loss of generality $A$ is of the form \eqref{tri}, real, and, along with \eqref{2}:
 \begin{equation}\label{123} \begin{split}
1-(a_{12}^2+a_{14}^2+a_{24}^2)+2\eta a_{12}{a_{14}}a_{24} & =0,\\
1-(a_{13}^2+a_{14}^2+a_{34}^2)+2\eta a_{13}{a_{14}}a_{34} & =0,\\
1-(a_{23}^2+a_{24}^2+a_{34}^2)+2\eta a_{23}{a_{24}}a_{34} & =0.
\end{split}\end{equation}

Since $A$ is real, the third flat portion of $\partial F(A)$, if it exists, must be vertical. Indeed, otherwise its reflection with respect to the real axis would be the fourth flat portion of
$\partial F(A)$, implying by Theorem~\ref{4x4} the unitary reducibility of $A$. But this would contradict Proposition~\ref{p:red}.

So, suppose now that a vertical flat portion of $\partial F(A)$ is indeed present. Then its abscissa, which it is convenient for us to denote by $-x/2$, is a multiple eigenvalue of $\re A$. Equivalently, the matrix
\[ \begin{bmatrix} x & a_{12} & a_{13} & a_{14} \\ a_{12} & x & a_{23} & a_{24} \\ a_{13} & a_{23} & x & a_{34} \\ a_{14} & a_{24} & a_{34} & x \end{bmatrix} \] has rank  at most two. Consequently,
\begin{equation}\label{456} \begin{split}
x^3-x({a_{13}}^2+{a_{12}}^2+{a_{23}}^2)+2a_{12}{a_{13}}a_{23} & =0,\\
x^3-x({a_{12}}^2+{a_{14}}^2+{a_{24}}^2)+2a_{12}{a_{14}}a_{24} & =0,\\
x^3-x({a_{13}}^2+{a_{14}}^2+{a_{34}}^2)+2a_{13}{a_{14}}a_{34} & =0,\\
x^3-x({a_{23}}^2+{a_{24}}^2+{a_{34}}^2)+2a_{23}{a_{24}}a_{34} & =0.
\end{split}\end{equation}
Comparing the respective lines in \eqref{2}--\eqref{123} and \eqref{456}, we conclude that $x^3-x+2C(1-x\eta)=0$ when in place of $C$ is plugged in any of the products $a_{12}{a_{13}}a_{23},a_{12}{a_{14}}a_{24},a_{13}{a_{14}}a_{34},a_{23}{a_{24}}a_{34}$. If $1- x \eta =0$, then $-x/2=-1/2 \eta$, the point on the negative x-axis where the support lines containing the flat portions of $\partial F(A)$ intersect. This would be a corner point on $F(A)$, thus implying unitary reducibility of $A$. So, $1- x \eta \ne 0$ and the above mentioned permissible values of $C$ are all equal:
\eq{epro} a_{12}{a_{13}}a_{23}=a_{12}{a_{14}}a_{24}=a_{13}{a_{14}}a_{34}=a_{23}{a_{24}}a_{34}. \en From here we conclude that the coefficients of $x$ in all the equations \eqref{456} are also equal:
\eq{equa} {a_{13}}^2+{a_{12}}^2+{a_{23}}^2={a_{12}}^2+{a_{14}}^2+{a_{24}}^2={a_{13}}^2+{a_{14}}^2+{a_{34}}^2={a_{23}}^2+{a_{24}}^2+{a_{34}}^2.\en From \eqref{epro}--\eqref{equa} we conclude that
\eq{sol} a_{34}=\epsilon_1 a_{12}, \ a_{14}=\epsilon_2 a_{23},\ a_{24}=\epsilon_3 a_{13}, \quad \epsilon_j=\pm 1, \ j=1,2,3, \en and either $a_{12}a_{13}a_{23}=0$ or $\epsilon_1=\epsilon_2=\epsilon_3$.
Since the former case is covered by Theorem~\ref{GauWuCircular}, we may concentrate on the latter.
Moreover, applying a unitary similarity via the diagonal matrix $\diag[1,1,1,-1]$, we may change the signs of all $\epsilon_j$ simultaneously, and thus to suppose that they are all equal to $1$.

In other words, it remains to consider \eq{Ar1}   A=\begin{bmatrix}0 & a_{1} & a_2 & a_{3} \\ 0 & 0 & a_{3} & a_2 \\  0 & 0 & 0 & a_{1} \\ 0 & 0 & 0 & 0\end{bmatrix} \en
 with $a_j\in\R$ and $a_1\neq 0$. By Lemma~\ref{realA}, the numerical range of a matrix of this form can only have a vertical flat portion on its boundary if that is the only flat portion.
 Therefore the original matrix has only the two flat portions on lines which are equidistant from the origin.
 \end{proof}

 \section{Proof of the main result}\label{s:proof}
 In this section, we show there are at most two flat portions on the boundary of the numerical range of a $4$-by-$4$ unitarily irreducible nilpotent matrix. This result follows from an analysis of the singularities of the polynomial $p_A$ defined in \eqref{pa}, so some facts about algebraic curves are reviewed next for convenience.

Define the complex projective plane $\mathbb{C} \mathbb{P}^2$ to be the set of all equivalence classes of points in $\mathbb{C}^3 -\left\{ (0,0,0)  \right\}$ determined by the equivalence relation $\sim$ where
$(x,y,z) \sim (a,b,c)$ if and only if $(x,y,z)=\lambda (a,b,c)$ for some nonzero complex number $\lambda$. The complex plane can be considered a subset of  $\mathbb{C} \mathbb{P}^2$ if
the point $(x,y,1)$ for $(x,y) \in \mathbb{R}^2$ is identified with $x+iy$.

An {\it algebraic curve} $C$ is defined to be the zero set of a homogeneous polynomial $f(x,y,z)$ in $\mathbb{C} \mathbb{P}^2$. If $f(x,y,z)$ has real coefficients, the real affine part of the curve C is defined to be all $(x,y) \in \mathbb{R}^2$ such that $f(x,y,1)=0$.

There is a one-to-one correspondence between points $(u,v,w) \in \mathbb{C} \mathbb{P}^2$ and lines $L$ in $\mathbb{C} \mathbb{P}^2$ given by the mapping
that sends $(u,v,w)$ to the line

$$\left\{ (x,y,z) \in \mathbb{C} \mathbb{P}^2 \, : \ ux+vy +wz=0 \right\}.$$
Therefore, $(u,v,w)$ could denote either a point or a line. In the latter case, the coordinates are called {\it line coordinates}.

Let $f(x,y,z)$ be a homogeneous polynomial. Let $C$ be the algebraic curve defined by $f$ in point coordinates. That is,

$$C=\left\{ (x,y,z) \in \mathbb{C} \mathbb{P}^2 \, : \, f(x,y,z)=0 \right\}.$$

The curve $C'$ which is {\it dual} to $C$ is obtained by considering $f$ in line coordinates. That is,

$$C'=\left\{ (u,v,w) \in \mathbb{C} \mathbb{P}^2 \, : \, ux+vy+wz=0 \text{ is a tangent line to $C$} \right\}.$$

The dual curve is an algebraic curve \cite{Bries} \cite{Walk} except in the special case where the original curve is a line and the dual is a point; we will assume this is not the case. Therefore there exists a homogeneous polynomial $p(u,v,w)$ such that $(u,v,w) \in C'$ if and only if $p(u,v,w)=0$.  The references just noted also show that the dual curve to the dual curve is the original curve. Therefore a point on $C$ corresponds to a tangent line to $C'$ and vice versa.

A {\it singular point} of a homogeneous polynomial $p$ is a point $(u_0,v_0,w_0)$  on the curve $C'$ defined by $p=0$ such that
$$(p_u(u_0,v_0,w_0),p_v(u_0,v_0,w_0), p_w(u_0,v_0,w_0))=(0,0,0).$$ If $(u_0,v_0,w_0)$ is not a singular point, then
the curve $C'$ has a well-defined tangent line at $(u_0,v_0,w_0)$ with line coordinates $(p_u(u_0,v_0,w_0),p_v(u_0,v_0,w_0),p_w(u_0,v_0,w_0))$.

If $(u_0,v_0,1)$ is a singular point of $C'$, then the line $u=u_0+\lambda t$, $v=v_0+ \mu t$ for $t \in \mathbb{C}$ intersects the curve $C'$ at $t=0$. If the second order partial derivatives of $p$ are not all zero at $(u_0,v_0,1)$, then the Taylor expansion of $p(u_0+\lambda t, v_0+\mu t,1)$ shows that

$$p(u,v,1)=\left(p_{xx}(u_0,v_0,1) \lambda^2+ 2p_{xy}(u_0,v_0,1) \lambda \mu + p_{yy}(u_0,v_0,1) \mu^2 \right) t^2 + \ldots ,$$
and in this case the curve $C'$ has two tangent lines (counting multiplicity) at $(u_0,v_0,1)$ defined by $(\lambda, \mu, 1)$ such that $p_{xx}(u_0,v_0,1) \lambda^2+ 2p_{xy}(u_0,v_0,1) \lambda \mu + p_{yy}(u_0,v_0,1) \mu^2=0$. If the minimum order for which all the partial derivatives at $(u_0, v_0, 1)$ of order $r$ are not identically zero is $r>2$, we similarly obtain $r$ tangent lines to $C'$ at $(u_0, v_0,1)$, counting multiplicities.
Conversely, any point at which $p=0$ has two (or more) tangent lines is a singular point of $p$.

Since $p_A(u,v,w)$ defined in \eqref{pa} is a homogeneous polynomial, its zero set is an algebraic curve in
$\mathbb{C} \mathbb{P}^2.$ Recall that Kippenhahn \cite{Ki} showed that the convex hull of the real affine part of the curve $C(A)$ which is dual to $p_A(u,v,w)=0$  is the numerical range of $A$. In terms of the description above, Kippenhahn showed that $p_A(u,v,w)$ is the polynomial $p$ above, while the curve $C(A)$ is given by $C$ above. He called $C(A)$ the {\it boundary generating curve} of $F(A)$.  In the proof below, $C'(A)$ is the curve given by $p_A(u,v,w)$ in point coordinates.

\begin{lemma}\label{singatflat} Let $A$ be an $n$-by-$n$ matrix. If the line
$$\left\{ (x,y) \in \mathbb{R}^2 \, : \, u_0x+v_0y+1=0 \right\}$$
contains a flat portion on the boundary of $F(A)$, then the homogeneous polynomial
$p_A(u,v,w)$ defined by equation \eqref{pa} has a singularity at $(u_0,v_0,1)$.
\end{lemma}

\begin{proof}  Any flat portion on the boundary of $F(A)$ is a line $L$ defined by
real numbers $u_0$, $v_0$ such that $p_A(u_0, v_0, 1)=0$. Furthermore, $L$ is tangent to two or more points on $C(A)$. Since the dual to the dual is the original curve, these points of tangency are both tangent lines to the dual curve $C'(A)$ at $(u_0,v_0,1)$. Therefore $(u_0,v_0,1)$ is a singular point of
$p_A$ since the tangent line there is not unique.
\end{proof}

Therefore the singularities of $p_A$ help determine how many flat portions are possible on the boundary of $F(A)$. In order to study the flat portions on the boundary of a general nilpotent $4 \times 4$ matrix, we will show that the associated polynomial $p_A$ has a special form where many of the coefficients are either zero or are equal to each other. The points at which singularities occur correspond to equations in a system of linear equations in the coefficients.

Note that if $z=\frac{u+iv}{2}$ and $p_A$ is given by \eqref{pa},  then $p_A(u,v,w)=\det(zA^*+\overline{z} A -(-w)I)$. The latter expression is $q(-w)$ where $q(w)$ is the characteristic polynomial of $zA^*+\overline{z} A$. Applying Newton's identities to this matrix yields the following lemma.

\begin{equation*}\label{evalsrotatedrealA1} \begin{split}
\lambda_1(\theta) & =\frac{1}{2}\left(-a_2-\sqrt{a_1^2+a_3^2-2a_1a_3 \cos\theta}\right),\\
\lambda_2(\theta) & =\frac{1}{2}\left(a_2-\sqrt{a_1^2+a_3^2+2a_1a_3 \cos\theta}\right),\\
\lambda_3 (\theta)& =\frac{1}{2}\left(-a_2+\sqrt{a_1^2+a_3^2-2a_1a_3 \cos\theta}\right),\\
\lambda_4(\theta) & =\frac{1}{2}\left(a_2+\sqrt{a_1^2+a_3^2+2a_1a_3 \cos\theta}\right).
\end{split}\end{equation*}

\begin{lemma}\label{nilpotentbgc} Let $A$ be a $4 \times 4$ nilpotent matrix. The boundary generating curve for $A$ is defined by
\begin{equation*}\begin{split} p_A(u,v,w)&=c_1 u^4+c_2 u^3 v+c_3 u^3w+ (c_1+c_4) u^2 v^2 +c_5 u^2 w^2\\+&c_6 u^2vw+c_2 uv^3+c_3uv^2w+c_4v^4+c_6v^3w+c_5 v^2w^2+w^4,\end{split}\end{equation*} where the coefficients $c_j$ are given below.

 \begin{align*}
 c_1&=-\frac{1}{16} \left( \tr(A^3A^*)+\tr(\left[A^*\right]^3 A)+ \tr(A^2 \left[A^*\right]^2)+\frac{1}{2} \tr(A^* A A^*A)-\frac{1}{2} \left[ \tr(A A^*) \right]^2 \right).\\ \notag
c_2&= \frac{i}{8} \left(\tr(A^3 A^*)- \tr(\left[A^*\right]^3 A) \right). \\ \notag
c_3&= \frac{1}{8} \left(\tr(\left[A^*\right]^2 A)+ \tr(A^2 A^*) \right) \\ \notag
c_4 &= \frac{1}{16} \left( \tr(A^3 A^*) + \tr(\left[A^*\right]^3 A) -\tr(A^2 \left[A^*\right]^2) -\frac{1}{2} \tr(A^* A A^* A) +\frac{1}{2}\left[ \tr(A A^*) \right]^2 \right). \\ \notag
c_5 &= -\frac{1}{4} \tr(A A^*) \\ \notag
c_6 &= \frac{i}{8} \left( \tr( \left[A^*\right]^2 A) -\tr(A^2 A^*)\right) .
\end{align*}

\end{lemma}

Proof of Lemma~\ref{nilpotentbgc}:  Let $M$ be an $n \times n$ matrix with characteristic polynomial $$q(w)= \sum_{j=0}^n (-1)^j q_j w^{n-j}.$$ By Newton's Identities (see \cite{Bake}), if $q_0=1$, then the remaining coefficients ($m=1, \ldots, n$) satisfy

\[ 
(-1)^m q_m=-\left( \frac{1}{m} \right) \sum_{j=0}^{m-1} (-1)^j  \tr \left( M^{m-j} \right) q_j.
\]

Applying these identities to $M=zA^*+\overline{z} A$ will yield the coefficients of the polynomial
$$q(w)=q_0 w^4-q_1 w^3+q_2 w^2-q_3 w+q_4$$ where each $q_j$ will be a polynomial in $u$ and $v$. The polynomial $p_A$ will then be defined by
$p_A(u,v,w)=q(-w)=q_0 w^4+q_1 w^3+q_2 w^2+q_3 w+q_4.$

Note that since $A$ is nilpotent, $\tr(A^k)=\tr([A^*]^k)=0$ for $k=1, 2, 3, 4$. The calculations below are also simplified with the identity $\tr (M N)=\tr(N M)$ for all $n \times n$ matrices $M$ and $N$.

Thus $q_0=1$ and $q_1=\tr(z A^* + \overline{z} A) q_0=0$. Next,

\begin{equation*}
q_2=-\frac{1}{2} \left( \tr([z A^* + \overline{z} A]^2) q_0-\tr[z A^* + \overline{z} A] q_1\right) =- \frac{1}{2} 2|z|^2 \tr(A A^*)= -\frac{1}{4} (u^2 +v^2) \tr(A A^*) .
\end{equation*}

\begin{align*}
q_3&=\frac{1}{3} \left( \tr([z A^* + \overline{z} A]^3) q_0-\tr([z A^* + \overline{z} A]^2) q_1+ \tr(z A^* + \overline{z} A) q_2 \right) \\ \notag
&=\frac{1}{3}\left( \tr([z A^* + \overline{z} A]^3) -\tr([z A^* + \overline{z} A]^2) 0 + 0 q_2 \right) \\ \notag
&=\frac{1}{3}   \tr([z A^* + \overline{z} A]^3) \\ \notag
&=\frac{1}{3} \left( z^3 \tr[A^*]^3+ z^2 \overline{z} \tr([A^*]^2A+ A [A^*]^2 + A^* A A^*)\right) \\ \notag
& +\frac{1}{3} \left( \overline{z}^2 z \tr( A A^* A + A^2 A^* + A^* A^2) + \overline{z}^3 \tr(A^3)\right) \\ \notag
&=z^2 \overline{z} \tr([A^*]^2A)+ \overline{z}^2 z \tr(  A^2 A^*) \\ \notag
&=\left(\frac{u^3+iu^2v +uv^2+iv^3}{8}\right)   \tr([A^*]^2A)+\left(\frac{u^3-iu^2v+uv^2-iv^3}{8} \right) \tr(  A^2 A^*) \\ \notag
&=\frac{u^3}{8} (\tr([A^*]^2A)+\tr(  A^2 A^*)) + \frac{u^2v}{8} (i   \tr([A^*]^2A)-i \tr(  A^2 A^*) )  \\ \notag
&+ \frac{uv^2}{8} (\tr([A^*]^2A)+\tr(  A^2 A^*))+ \frac{v^3}{8}  (i   \tr([A^*]^2A)-i \tr(  A^2 A^*) ). \notag
\end{align*}

Finally,

\begin{align*}
q_4&=-\frac{1}{4} \left\{ \tr([z A^* + \overline{z} A]^4 )q_0-\tr([z A^* + \overline{z} A]^3) q_1+ \tr([z A^* + \overline{z} A]^2) q_2 -\tr(z A^* + \overline{z} A) q_3 \right\} \\ \notag
&=-\frac{1}{4} \left\{ \tr([z A^* + \overline{z} A]^4) + \tr([z A^* + \overline{z} A]^2) q_2 \right\} \\ \notag
&=-\frac{1}{4} \left\{ \tr([z A^* + \overline{z} A]^4) + 2|z|^2 \tr(A A^*) (- |z|^2 \tr(A A^*) )  \right\} \\ \notag
&=-\frac{1}{4} \left\{  \tr([z A^* + \overline{z} A]^4 - 2 |z|^4 \left[ \tr(A A^*)  \right]^2 \right\} \\ \notag
&=-\frac{1}{4} \left\{ z^4 \tr([A^*]^4) + z^3  \overline{z} \ 4 \tr([A^*]^3 A) + |z|^4 ( 4 \tr([A^*]^2 A^2) + 2 \tr(A A^* A A^*) -2 [\tr(A A^*)]^2) \right. \\ \notag
&+ \left. \overline{z}^3 z \ 4 \tr(A^3 A^*) + \overline{z}^4 \tr(A^4)\right\}\\ \notag
&= -\frac{1}{64} \left\{ (u+iv)^3 (u-iv) 4 \tr([A^*]^3 A) +(u-iv)^3(u+iv)  4 \tr(A^3 A^*) \right.\\ \notag
&+ \left.(u^2+v^2)^2 ( 4 \tr([A^*]^2 A^2) + 2 \tr(A A^* A A^*)  -2 [\tr(A A^*)]^2) \right\} \\ \notag
&=  -\frac{1}{64} \left\{ (u^4+2 i u^3 v+ 2iu v^3-v^4) 4 \tr([A^*]^3 A) +(u^4-2 i u^3 v- 2iu v^3-v^4) 4 \tr(A^3 A^*) \right. \\ \notag
&+ \left. (u^4+ 2 u^2v^2 +v^4)( 4 \tr([A^*]^2 A^2) + 2 \tr(A A^* A A^*) -2 [\tr(A A^*)]^2)\right\} \\ \notag
\end{align*}
\begin{align*}
&=u^4 \left(\frac{- 2\tr([A^*]^3 A)- 2 \tr(A^3 A^*)-2\tr([A^*]^2 A^2) -\tr(A A^* A A^*)+[\tr(A A^*)]^2}{32}\right) \\ \notag
&+ u^3 v\left(\frac{-i \tr([A^*]^3 A)+i\tr(A^3 A^*) }{8}\right)+ u^2 v^2 \left(\frac{-2\tr([A^*]^2 A^2) -\tr(A A^* A A^*)+[\tr(A A^*)]^2}{16}\right) \\ \notag
&+ uv^3 \left(\frac{i \tr(A^3 A^*)-i\tr([A^*]^3 A) }{8}\right) \\ \notag
&+ v^4 \left(\frac{2\tr([A^*]^3 A)+2 \tr(A^3 A^*) -2\tr([A^*]^2 A^2) -\tr(A A^* A A^*)+[\tr(A A^*)]^2 }{32}\right). \notag
\end{align*}

Now $p_A(u,v,w)=w^4+0w^3+q_2 w^2+ q_3 w +q_4, $ and from this expression we can identity the coefficients of each of the degree 4 homogeneous terms in $u$, $v$, and $w$ as stated in the lemma.

The $w^4$ term has coefficient 1 and all of the terms containing $w^3$ have coefficient 0.

 The terms containing $w^2$ are obtained from $q_2$ and clearly the $u^2 w^2$ and $v^2w^2$ coefficients are both $c_5=-\tr(A A^*)/4,$ while there is no $u v w^2$ term.

 The terms containing $w$ are obtained from $q_3$. Note that the coefficients of $u^3 w$ and $uv^2 w$ are equal to each other with the value $c_3$, while the coefficients of $u^2vw$ and $v^3w$ are equal to each other with the value $c_6$.

For the terms without $w$, note that $c_1$ is the coefficient of $u^4$ in $q_4$ and $c_4$ is the coefficient of $v^4$ in $q_4$. In addition, the coefficient of $u^2 v^2$ in $q_4$ is exactly $c_1+c_4$. Finally, the coefficients of $u^3v$ and $uv^3$ are both equal to $c_2$. $\square$

Now we consider the condition where $p_A$ has a singularity.

\begin{lemma}\label{singularity} The homogeneous polynomial  $p_A$ from Lemma~\ref{nilpotentbgc}  has a singularity at $(u,v,w)$ if and only if

\begin{align}\label{gensings}
(4u^3+2 uv^2) c_1+ (3 u^2v+v^3) c_2+ (3u^2w+ v^2w) c_3+ (2 uv^2) c_4+ (2uw^2) c_5+ (2 uvw) c_6 &=0. \\ \notag
(2u^2v) c_1+(u^3+ 3uv^2) c_2+ (2uvw )c_3+ (2u^2v+4v^3) c_4+ (2vw^2) c_5 +(u^2w+3v^2w)c_6&=0 \\ \notag
(u^3+uv^2) c_3+ (2u^2w+2v^2w)c_5+(u^2v+v^3) c_6+4w^3&=0\notag
\end{align}
\end{lemma}

\begin{proof} By Lemma~\ref{singatflat} he polynomial $p_A$ has a singularity at a point $(u,v,w)$ if and only if
$$\frac{\partial}{\partial u} p_A(u,v,w)=\frac{\partial}{\partial v} p_A(u,v,w)=\frac{\partial}{\partial w} p_A(u,v,w)=0.$$

Using the form of $p_A$ from Lemma~\ref{nilpotentbgc} yields the equations in \eqref{gensings}.
\end{proof}

When $w=1$, the system \eqref{gensings} becomes the non-homogeneous system

\begin{align}\label{affinesystem}
(4u^3+2 uv^2) c_1+ (3 u^2v+v^3) c_2+ (3u^2+v^2) c_3+ (2 uv^2) c_4+ (2u) c_5+ (2 uv) c_6 &=0. \\ \notag
(2u^2v) c_1+(u^3+ 3uv^2) c_2+ (2uv )c_3+ (2u^2v+4v^3) c_4+ (2v) c_5 +(u^2+3v^2)c_6&=0. \\ \notag
(u^3+uv^2) c_3+ (2u^2+2v^2)c_5+(u^2v+v^3) c_6&=-4,\notag
\end{align}
from which the following special case is immediate.

\begin{lemma}\label{singularityspec} The polynomial $p_A$ has a singularity at $(2,0,1)$ if and only if

\begin{align*}
8c_1+ 3c_3+c_5 &=0. \\ \notag
2c_2+c_6&=0. \\ \notag
8c_3+8c_5&=-4. \notag
\end{align*}
\end{lemma}

The above condition is necessary for $F(A)$ to have a flat portion at $x=-1/2$. This system can be rewritten as

\begin{equation}\label{cs}
\left\{
\begin{aligned}
c_3 &=-4c_1+\frac{1}{4}. \\
c_6&=-2c_2. \\
c_5&=4c_1-\frac{3}{4}.
\end{aligned}
\right.
\end{equation}

We can use this necessary condition to eliminate certain possibilities involving other flat portions.

\begin{theorem}\label{twoflat} If $A$ is a 4-by-4 unitarily irreducible nilpotent matrix, then $\partial F(A)$ has at most two flat portions.
\end{theorem}

\begin{proof} Assume $A$ is a 4-by-4 unitarily irreducible nilpotent matrix. The associated polynomial $p_A$ thus has the form given by Lemma~\ref{nilpotentbgc}. If there is at least one flat portion on the boundary of $F(A)$, we may rotate and scale $A$ so that there is a flat portion on the line $x=-\frac{1}{2}$. This flat portion corresponds to a singularity $(2,0,1)$ so the system \eqref{cs} is satisfied. Thus only the variables $c_1$, $c_2$, and $c_4$ are free.

Assume there is another flat portion on the line $ux+vy+1=0$. By Lemma~\ref{singatflat} there is a singularity at this $(u,v,1)$ where $(u,v) \ne (2,0)$. For any such singularity, we can eliminate $c_3$, $c_5$ and $c_6$ in the necessary equations \eqref{affinesystem} to obtain the new consistent system below.

\begin{equation}\label{gensecondsing}
\left\{
\begin{aligned}
(4u^3+2 uv^2+8u-12u^2-4v^2) c_1+ (3 u^2v+v^3-4uv) c_2+  (2 uv^2) c_4 &=-\frac{1}{4}v^2+\frac{3}{2}u-\frac{3}{4}u^2. \\
(2u^2v-8uv+8v) c_1+ (u^3+ 3uv^2-2u^2-6v^2) c_2+ (2u^2v+4v^3) c_4&=-\frac{1}{2}uv+\frac{3}{2}v. \\
4(2-u)(u^2+v^2) c_1-2v (u^2+v^2)c_2&=-4+\left(\frac{6-u}{4}\right)(u^2+v^2).
\end{aligned}
\right.
\end{equation}

If $v=0$ and $u \ne 2$ for the singular point $(u,v,1)$, then the corresponding flat portion is on a vertical line $x=-\frac{1}{u}$ and there are two parallel flat portions which must be the only flat portions by Corollary~\ref{co:onlytwo}. If $u=0$ at the singularity, then the system above is consistent if and only if $v= \pm 2$. The point $(0, 2)$ could only correspond to a flat portion on the line $y=- \frac{1}{2}$ and the point $(0,-2)$ could only correspond to a flat portion on $y=\frac{1}{2}$. Each of these support lines
is at a distance of $\frac{1}{2}$ from the origin. Therefore Proposition~\ref{p:nonpa} shows that if there are flat portions both on $x=-\frac{1}{2}$ and on either $y=\frac{1}{2}$ or $y=-\frac{1}{2}$, then there will only be these two flat portions. Therefore in the remainder of the argument, we will assume that any singular points satisfy $u \ne 0$ and $v \ne 0$.

To simplify row reductions in \eqref{gensecondsing}, put $c_4$ in the first column and $c_1$ in the third column. If the resulting matrix is row reduced using only the extra assumption that neither $u$ nor $v$ is zero then we get
the matrix

$$\left(
\begin{array}{cccc}
 2 u v^2 & v^3+u (3 u-4) v & 2 (u-2) \left(v^2+2 (u-1) u\right) & -\frac{v^2}{4}-\frac{3}{4} (u-2) u \\
 0 & -2 v & 8-4 u & -\frac{u}{4}+\frac{3}{2}-\frac{4}{u^2+v^2} \\
 0 & -2v(u^2+v^2-u) & 4(2-u)(u^2+v^2-u) & \frac{3}{4}u^2+\frac{1}{2} v^2-\frac{3}{2}u \\
\end{array}
\right).$$
If $u^2+v^2-u=0$, the system described above is inconsistent unless $\frac{3}{4}u^2+\frac{1}{2} v^2-\frac{3}{2}u=0$, but the combination of those equations implies $u=0$ which has already been ruled out. Therefore we may assume $u^2+v^2-u \ne 0$ and thus obtain the row-equivalent matrix

\begin{equation}\label{threebyfour}\left(
\begin{array}{cccc}
 2 u v^2 & v^3+u (3 u-4) v & 2 (u-2) \left(v^2+2 (u-1) u\right) & -\frac{v^2}{4}-\frac{3}{4} (u-2) u \\
 0 &  v & 2(u-2) & \frac{u}{8}-\frac{3}{4}+\frac{2}{u^2+v^2} \\
 0 & 0 & 0 & \frac{\left(u^2+v^2-4\right) \left(u (u-2)^2+(u-4) v^2\right)}{4 \left(u^2+v^2\right)} \\
\end{array}
\right).
\end{equation}

The matrix \eqref{threebyfour} corresponds to an inconsistent system unless either $u^2+v^2=4$ or $v^2=\frac{u(u-2)^2}{4-u}$.  Any point $(u,v)$ corresponding to a flat portion must satisfy at least one of these conditions so if there are two flat portions besides the one on $x=-\frac{1}{2}$, both must satisfy at least one of these conditions. When $u^2+v^2=4$ the line $ux+vy+1=0$ is a distance of $\frac{1}{2}$ from the origin, which is the same as the line $x=-\frac{1}{2}$ containing the original flat portion. Consequently if there is a singularity with $u^2+v^2=4$, then the corresponding line contains the only other possible flat portion by Proposition~\ref{p:nonpa}.

Therefore, there could only be three flat portions if there are two different pairs $(u_1, v_1)$ and $(u_2, v_2)$ that form a $6 \times 4$ augmented matrix that is consistent and where each $(u_j, v_j)$ pair satisfies $v^2=\frac{u(u-2)^2}{4-u}$.

For a given singularity $(u,v,1)$ with $v^2=\frac{u(u-2)^2}{4-u}$, lengthy calculations show that the matrix \eqref{threebyfour} is row equivalent to

$$\left(
\begin{array}{cccc}
 1 & 0 &\frac{u-4}{u} & \frac{(u-6)(u-4)}{8u^2} \\
 0 & 1 & \frac{2 (u-2)}{v} & \frac{(u-2)(u-8)}{8 u v} \\
 0 & 0 &0 &0\\
\end{array}
\right).$$

Therefore, if there are three flat portions on $\partial F(A)$, then there exist points $(u_1,v_1)$ and $(u_2, v_2)$ with neither $u_i=0$ nor $v_i=0$ for $i=1,2$ such that
the matrix $M$ below corresponds to a consistent system, and consequently satisfies
$\det(M)=0$.

$$M=\left(
\begin{array}{cccc}
 1 & 0 &\frac{u_1-4}{u_1} & \frac{(u_1-6)(u_1-4)}{8u_1^2} \\
 0 & 1 & \frac{2 (u_1-2)}{v_1} & \frac{(u_1-2)(u_1-8)}{8 u_1 v_1} \\
 1 & 0 &\frac{u_2-4}{u_2} & \frac{(u_2-6)(u_2-4)}{8u_2^2} \\
 0 & 1 & \frac{2 (u_2-2)}{v_2} & \frac{(u_2-2)(u_2-8)}{8 u_2 v_2} \\
\end{array}
\right).$$

Note that $M$ has the form

$$M=\left(
\begin{array}{cccc}
 1 & 0 &a_1 & a_1c_1 \\
 0 & 1 & b_1 & b_1 d_1 \\
 1 & 0 &a_2 & a_2c_2 \\
 0 & 1 & b_2 & b_2d_2 \\
\end{array}
\right),$$
from which it follows that

$$\det(M)=b_2\left( (a_2-a_1)d_2+a_1c_1-a_2c_2\right)+b_1\left( (a_1-a_2)d_1+a_2c_2-a_1c_1\right).$$

Therefore,
\begin{align*}
\det(M)&= \frac{2 (u_2-2)}{v_2} \left(\left(\frac{u_2-4}{u_2}-\frac{u_1-4}{u_1}\right) \frac{(u_2-8)}{16 u_2 } + \frac{(u_1-6)(u_1-4)}{8u_1^2}- \frac{(u_2-6)(u_2-4)}{8u_2^2}\right)\\
&+
\frac{2 (u_1-2)}{v_1}\left( \left(\frac{u_1-4}{u_1}-\frac{u_2-4}{u_2}\right) \frac{(u_1-8)}{16 u_1 } + \frac{(u_2-6)(u_2-4)}{8u_2^2}- \frac{(u_1-6)(u_1-4)}{8u_1^2}\right).
\end{align*}

Simplifying and removing the common factor $\frac{(u_1-u_2)}{4 u_1^2 u_2^2}$ from both terms in parentheses shows that

\begin{align*}
\det(M)&=\frac{(u_1-u_2)}{4 u_1^2 u_2^2} \left[ \frac{2 (u_2-2)}{v_2} \left(-u_1(u_2-8)-12u_1-12u_2+5u_1u_2\right) \right.\\
&+\left. \frac{2 (u_1-2)}{v_1}\left(u_2(u_1-8)+12u_1+12u_2-5u_1u_2\right)\right].
\end{align*}

If $u_1=u_2$, then $v_1^2=\frac{u_1(u_1-2)^2}{4-u_1}=\frac{u_2(u_2-2)^2}{4-u_2}=v_2^2$ and hence the singular points $(u_1, v_1)$ and $(u_1, v_2)$ result in flat portions that are the same
distance $1/\sqrt{u_1^2+v_1^2}$ from the origin. Therefore these two flat portions cannot coincide with the original flat portion at $x=-1/2$ by Proposition~\ref{p:nonpa}.
So the only remaining case that could lead to three flat portions on the boundary of $F(A)$ is if $\det(M)=0$ because

\begin{equation}\label{deteq}
 \frac{(u_2-2)}{v_2}\left(u_1u_2-u_1-3u_2\right)= \frac{ (u_1-2)}{v_1}\left(u_1u_2-u_2-3u_1\right).
\end{equation}

Squaring both sides of \eqref{deteq} and replacing $ \frac{(u_j-2)^2}{v_j^2}$ with
$\frac{4-u_j}{u_j}$ for $j=1,2$ results in

$$\frac{(4-u_2)}{u_2}\left(u_1u_2-u_1-3u_2\right)^2= \frac{ (4-u_1)}{u_1}\left(u_1u_2-u_2-3u_1\right)^2,$$
and this implies that

$$u_1(4-u_2)(u_1u_2-u_1-3u_2)^2-u_2(4-u_1)(u_1u_2-u_2-3u_1)^2=0.$$
However, the left side of the expression above is $4(u_1-u_2)^3$, and as mentioned previously, $u_1=u_2$ leads to a contradiction of Proposition~\ref{p:nonpa}.
\end{proof}

\section{Case of unitarily reducible 5-by-5 matrices} \label{s:5x5}

With Theorem~\ref{twoflat} at our disposal, it is not difficult to describe completely the situation with the flat portions on the boundary of $F(A)$ for nilpotent $5$-by-$5$ matrices $A$, provided
that they are unitarily reducible.

\begin{theorem}\label{th:5x5red} Let a $5$-by-$5$ matrix $A$ be nilpotent and unitarily reducible. Then there are at most two flat portions on the boundary of its numerical range. Moreover, any number from $0$ to $2$
is actually attained by some such matrices $A$. \end{theorem}

\begin{proof} Suppose first that $\ker A\cap\ker A^*\neq \{0\}$. Then $A$ is unitarily similar to $A_1\oplus [0]$, where $A_1$ is also nilpotent. Consequently, $F(A)=F(A_1)$, and the statement follows from
Proposition~\ref{p:red} if $A_1$ is in its turn unitarily reducible and Theorem~\ref{twoflat} otherwise. Note that all three possibilities (no flat portions, one or two flat portions on $\partial F(A)$ already materialize
in this case.

Suppose now that $\ker A\cap\ker A^*=\{0\}$. Then the only possible structure of matrices unitarily similar to $A$ is $A_1\oplus A_2$, with one $2$-by-$2$ and one $3$-by-$3$ block. Multiplying $A$ by an appropriate
scalar and applying yet another unitary similarity if needed, we may without loss of generality suppose that
\[ A_1= \begin{pmatrix} 0 & r \\ 0 & 0 \end{pmatrix}, \quad  A_2= \begin{pmatrix} 0 & r_1 & r_2 \\ 0 & 0 & r_3 \\ 0 & 0 & 0 \end{pmatrix},\] where $r,r_1,r_3>0$, $r_2\geq 0$. The numerical range of $A_1$ is the
circular disk of the radius $r/2$ centered at the origin. If $r_2=0$, then $F(A_2)$ also is a circular disk centered at the origin  \cite[Theorem~4.1]{KRS}, and $F(A)$, being the largest of the two disks, has no flat portions
on its boundary. So, it remains to consider the case when all $r_j$ are positive.

The distance from the origin to the supporting line of $F(A_2)$ forming angle $\theta$ with the vertical axis equals the maximal eigenvalue of $\re e^{i\theta}A$, that is, half of the largest root of the polynomial
\eq{f} f_\theta(\lambda)=\lambda^3-\lambda (r_1^2+r_2^2+r_3^2)-2r_1r_2r_3\cos\theta. \en
Since $f_\theta$ is a monotonically increasing function of $\lambda$ for $\lambda\geq\lambda_0=\left(\frac{r_1^2+r_2^2+r_3^2}{3}\right)^{1/2}$, and since $f_\theta(\lambda_{0})\leq 0$ due to the inequality between the arithmetic and geometric means of $r_j^2$, the maximal root $\lambda(\theta)$ of $f_\theta$ is bigger than $\lambda_0$. If $\cos\theta_1<\cos\theta_2$, then $f_{\theta_1}(\lambda(\theta_2))>0$, and so
$\lambda(\theta_1)<\lambda(\theta_2)$. In other words, the maximal root of $f_\theta$ is a strictly monotonic function of $\theta$ both on $[0,\pi]$ and $[-\pi,0]$. So, the disk $F(A_1)$ will have exactly two common supporting lines with $F(A_2)$ when $r/2$ lies strictly between the minimal and maximal distance from the points of $\partial F(A_2)$ to the origin, and none otherwise. Further reasoning depends on whether or not the parameters $r_j$ are all equal.

{\sl Case 1.} Among $r_j$ at least two are distinct. According to already cited Theorem~4.1 from \cite{KRS}, $F(A_2)$ has the so called ``ovular shape''; in particular, there are no flat portions on its boundary.
Then the flat portions on the boundary of $F(A)$ are exactly those lying on common supporting lines of $F(A_1)$ and $F(A_2)$, and so there are either two or none of them. To be more specific, the distance from the origin to the supporting line at $\theta$ discussed above is (using Vi\`ete's formula)

\begin{equation*}
\lambda(\theta)/2=\sqrt{s/3} \cos \left( \frac{1}{3} \arccos \left( \frac{3 \sqrt{3} t \cos\theta}{s^{3/2}} \right)  \right),
\end{equation*}
where $s=r_1^2+r_2^2+r_3^2$ and $t=r_1r_2r_3.$ Since the distance from the origin to the tangent line of the disk $F(A_1)$ is a constant $r/2$,  there will be two  values of $\theta$ (opposite of each other) for which these  tangent lines coincide with supporting lines of $F(A_2)$ if and only if

\begin{equation*} \sqrt{s/3} \cos \left( \frac{1}{3} \arccos \left( -\frac{3 \sqrt{3} t }{s^{3/2}} \right)  \right)<\frac{r}{2}<\sqrt{s/3} \cos \left( \frac{1}{3} \arccos \left( \frac{3 \sqrt{3} t }{s^{3/2}} \right)  \right),
\end{equation*}
and none otherwise.

{\sl Case 2.} All $r_j$ are equal. The boundary generating curve $C(A_2)$ (see Section~\ref{s:proof} for the definition) is then a cardioid, appropriately shifted and scaled, as shown (yet again) in \cite[Theorem~4.1]{KRS}.
Consequently, $\partial F(A_2)$ itself has a (vertical) flat portion, and we need to go into more details. To this end, suppose (without loss of generality) that $r_1=r_2=r_3=3$, and invoke formula on p. 130 of \cite{KRS}, according to which $C(A_2)$ is given by the parametric equations \eq{car}  x(\theta)=2\cos\theta+\cos 2\theta, \quad y(\theta) = 2\sin\theta+\sin 2\theta, \qquad \theta\in [-\pi,\pi]. \en
The boundary of $F(A_2)$ is the union of the arc $\gamma$ of \eqref{car} corresponding to $\theta\in [-2\pi/3,2\pi/3]$ with the vertical line segment $\ell$ connecting its endpoints. The remaining portion of the curve \eqref{car}
lies inside $F(A_2)$.

Observe also that $\abs{x(\theta)+iy(\theta)}=\sqrt{5+4\cos\theta}$ is an even function of $\theta$ monotonically decreasing on $[0,\pi]$. Putting these pieces together yields the following:

For $r\leq 3$, the disk $F(A_1)$ lies inside $F(A_2)$.  Thus, $F(A)=F(A_2)$ has one flat portion on the boundary.

For $3<r\leq 2\sqrt{3}$ the circle $\partial F(A_1)$ intersects $\partial F(A_2)$ at two points of $\ell$. This results in two flat portions on $\partial F(A)$.

For $2\sqrt{3}<r<6$ the circle $\partial F(A_1)$ intersects $\partial F(A_2)$ at two points of $\gamma$, while $\ell$ lies inside $F(A_1)$. This again results in two flat portions on $\partial F(A)$.

Finally, if $r\geq 6$, then $F(A_2)$ lies in $F(A_1)$, so $F(A)=F(A_1)$ is a circular disk, and there are no flat portions on its boundary.

\end{proof}

The case where $r_1=r_2=r_2=3$ and $r=3.3$, which results in two flat portions caused by intersections between the circular disk and the numerical range of the $3$-by-$3$ nilpotent matrix, is shown in Figure~\ref{5by5twoflat}. The case where $r_1=r_2=r_3=r=3.3$, which results in one flat portion from the numerical range of the $3$-by-$3$ block, with the circular disk tangent inside is shown in Figure~\ref{5by5oneflat}.  In all cases the cardiod boundary generating curve and the boundary of the numerical range of the $2$-by-$2$ matrix is included.

\begin{figure}[h]
\centering
\includegraphics[scale=.4]{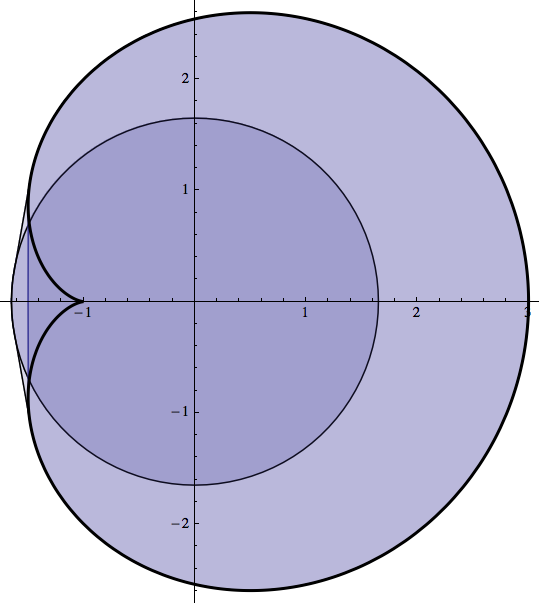}
\caption{Numerical range of reducible $5$-by-$5$ matrix with two flat portions}\label{5by5twoflat}
\end{figure}

\begin{figure}[h]
\centering
\includegraphics[scale=.4]{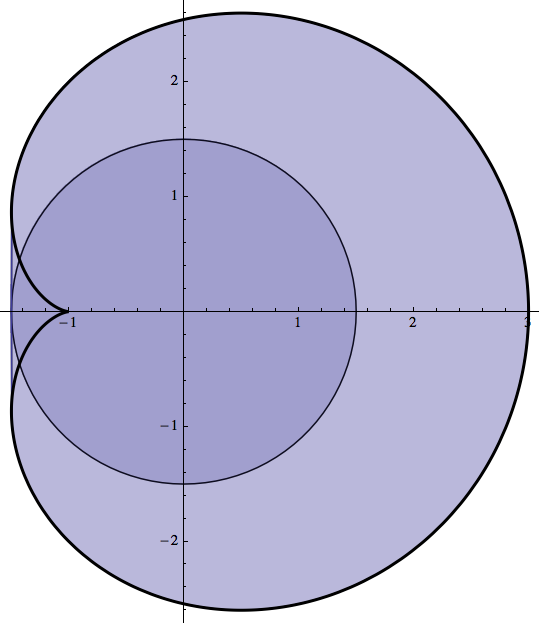}
\caption{Numerical range of reducible $5$-by-$5$ matrix with one flat portion}\label{5by5oneflat}
\end{figure}

\clearpage

\providecommand{\bysame}{\leavevmode\hbox to3em{\hrulefill}\thinspace}
\providecommand{\MR}{\relax\ifhmode\unskip\space\fi MR }
\providecommand{\MRhref}[2]{%
  \href{http://www.ams.org/mathscinet-getitem?mr=#1}{#2}
}
\providecommand{\href}[2]{#2}

 \end{document}